\documentclass[10pt]{article}

\pdfoutput=1 
\usepackage{etex}
\usepackage[french, english]{babel} 
\usepackage[utf8]{inputenc}
\usepackage[T1]{fontenc}
\usepackage{lmodern}
\usepackage{authblk}

\usepackage{amsmath, latexsym, amssymb, amsfonts, amsthm}
\numberwithin{equation}{section}
\usepackage{mathtools}
\usepackage{graphicx}
\usepackage{subcaption}
\usepackage{enumitem}
\usepackage{multicol}
\usepackage[colorlinks,citecolor=black,filecolor=black,linkcolor=black,urlcolor=black ]{hyperref}

\def\be{\begin{equation}}   \def\ee{\end{equation}}
\def\ba   {\begin{array}}      \def\ea   {\end{array}}
\def\bea  {\begin{eqnarray}}   \def\eea  {\end{eqnarray}}
\def\bean {\begin{eqnarray*}}  \def\eean {\end{eqnarray*}}

\newtheorem{theorem} {Theorem}
\newtheorem{lemma}{Lemma}
\newtheorem{definition} {Definition}

\theoremstyle{definition}

\newcommand{\im}[1]{\ensuremath{\mathrm{\mathbf{i}}_{\mathbf{#1}}}}

\newcommand{\R}{\ensuremath{\mathbb{R}}}

\newcommand{\suite}{\{P_{c}^{(n)}(0)\}_{n\in\mathbb{N}}}

\newcommand{\vspan}{\text{span}_{\R}}

\usepackage{todonotes}

\newcommand{\spher}{\mathcal{M}_{\textnormal{Im}\mathbb{H}}}

\DeclareMathOperator{\arctantwo}{arctan2}

\title{On the Algebraic Foundation of the Mandelbulb}

\author{Vanessa Boily\thanks{E-mail: {\tt vanessa.boily@uqtr.ca}} }
\author{Dominic Rochon\thanks{E-mail: {\tt dominic.rochon@uqtr.ca}}}
\affil{Département de mathématiques et d'informatique, \\ Université du Québec à Trois-Rivières, C.P. 500, Trois-Rivières, Québec, Canada, G9A 5H7.}
\date{\small{}}

\begin{document}
\maketitle

\begin{abstract}
In this paper, we generalize the Mandelbrot set using quaternions and spherical coordinates. In particular, we use pure quaternions to define a spherical product. This product, which is inspired by the product of complex numbers, add the angles and multiply the radii of the spherical coordinates. We show that the algebraic structure of pure quaternions with the spherical product is a commutative unital magma. Then, we present several generalizations of the Mandelbrot set. Among them, we present a set that is visually identical to the so-called Mandelbulb. We show that this set is bounded and that it can be generated by an escape time algorithm. We also define another generalization, the bulbic Mandelbrot set. We show that one of its 2D cuts has the same dynamics as the Mandelbrot set and that we can generate this set only with a quaternionic product, without using the spherical product. 
\end{abstract}\vspace{0.5cm}
\noindent\textbf{AMS subject classification:} 32A30, 30G35, 00A69
\\
\textbf{Keywords:} Spherical dynamics, Generalized Mandelbrot sets, Quaternions, Bulbic Mandelbrot set, Mandelbulb, Goldenbulb, Mandeldart, 3D Fractals
\section*{Introduction}
Quaternions were discovered by Sir William Rowan Hamilton while he was trying to extend the complex numbers to a third dimension by adding imaginary units. These four-dimensional numbers form a noncommutative field. In particular, quaternions can be used to rotate vectors around an arbitrary vector in the 3D space \cite{hanson, kantor, koecher}. 

In this paper, we are interested in the 3D generalizations of the Mandelbrot set by using the concept of rotation in the space. The Mandelbrot set is generated by iterating a complex quadratic polynomial and has been studied for many years \cite{beardon, douady, falconer}. Iterations of quaternionic polynomials have also been used to generalized the Mandelbrot set in 3D. However, since quaternions have four dimensions, the visualization of the set is only possible through 3D slices. Even if Bedding and Briggs said that possibly no interesting dynamics occur in the case of the quaternionic Mandelbrot set, it has been studied in many articles over the years (see \cite{bedding, cheng, dang, katunin, wang}).

Spherical coordinates were also used to generalize the Mandelbrot set in 3D. This method was inspired by the geometric result of squaring a complex number that is squaring the radius and doubling the angle. The famous Mandelbulb fractal is the result of this generalization \cite{Barrallo, CNRS, nylander, quilez, white}.

Although the use of spherical coordinates makes it possible to obtain beautiful fractals like the Mandelbulb, there is no specific algebraic structure behind it. In fact, these fractals were generated using a ray tracing method inspired by other methods \cite{brouillettepariserochon, dang, rochonmartineau}. In this paper, we present an algebraic structure supported by pure quaternions to generate fractals using a spherical polynomial. Moreover, we show that our spherical Mandelbrot set is bounded by a sphere of radius 2. We also establish a relationship between a variation of this generalized Mandelbrot set and a 3D slice of the quaternionic Mandelbrot set.

In Section \ref{sec:preliminaries}, quaternions and spherical coordinates are introduced. In Section \ref{sec:spherical}, we define the spherical product of pure quaternions and present some of its properties. Then, in Section \ref{sec:quaternionic}, we present a generalization of the Mandelbrot set, the quaternionic Mandelbrot set. Finally, in Section \ref{sec:mandelbrot}, we introduce a generalization of the Mandelbrot set in true 3D using our spherical product of pure quaternions. For the power 8, we show that this set is visually the same as the so-called Mandelbulb and we found its specific bound in the space.

%
%
%
%
\section{Definitions and basics}\label{sec:preliminaries}

\subsection{The quaternions}
We present here a summary about quaternions and some of their properties. 
Quaternions are an extension of complex numbers. Complex numbers have two real components and an imaginary unit. Quaternions have four real components and three imaginary units such that $i^2=j^2=k^2=ijk=-1$. 
The quaternions set is noted
$$
\mathbb{H}:=\lbrace a+bi+cj+dk\: |\: a,b,c,d \in \mathbb{R}\rbrace.
$$
The quaternions set is an associative, noncommutative division algebra. Addition is defined the same way as vector addition.To define a multiplication rule, we need to assign values to $i,j$ and $k$ when multiplied two by two. These values are: 
\begin{align*}
ij=k,\ \  ji=-k,\\
jk=i,\ \  kj=-i,\\
ki=j,\ \  ik=-j.
\end{align*}
Multiplication of quaternions is defined in this way:
\begin{align*}
q_1q_2&:=(a_1a_2-b_1b_2-c_1c_2-d_1d_2)+(a_1b_2+b_1a_2+c_1d_2-d_1c_2)i\\
&+(a_1c_2+c_1a_2+d_1b_2-b_1d_2)j+(a_1d_2+d_1a_2+b_1c_2-c_1b_2)k.
\end{align*}
Analogously to complex numbers, quaternions have a modulus. This modulus is the Euclidean norm. The modulus is denoted $\|q\|:=\sqrt{a^2+b^2+c^2+d^2}$. An important property of quaternions is that the modulus of a product is the product of the modulus, which means that we have $\|q_1q_2\|=\|q_1\|\|q_2\|$ (see \cite{kantor}).

Like the complex numbers, quaternions have a polar representation. Consider a quaternion $q=a+\textbf{q}$ where $\textbf{q} \neq 0$ is the vector parts of $q$ which means $\textbf{q}=bi+cj+dk$. Thus, there exist a unique angle $0 \leq \phi \leq \pi$ such that
$$
q = \| q\| ( \cos\phi + \mathbf{p} \sin\phi )
$$
where $\textbf{p}=\frac{\mathbf{q}}{\| \textbf{q} \|}$ is a unit vector.

The following theorem is a generalization of the De Moivre formula for the complex numbers called the quaternionic De Moivre formula \cite{kantor}.
\begin{theorem}
\label{theo:demoivre}
If $q=\| q\| ( \cos\phi + \mathbf{p} \sin\phi )$ is a quaternion in its polar representation and $n \in \mathbb{N}$, then
$$
q^n=\|q\|^n\left[\cos (n\phi)+\mathbf{p}\sin (n\phi)\right].
$$
\end{theorem}
There is an important subset in the quaternions set that is called the pure quaternions set and is denoted \textnormal{Im}$\mathbb{H}$. A pure quaternion is a quaternion $q=a+bi+cj+dk$ where $a=0$.
Pure quaternions can be used to rotate vectors around an arbitrary vector in 3D space using the following theorem \cite{hanson, kantor}.
\begin{theorem}
\label{theo:rotation}
Consider $P$ and $u$ two pure quaternions and a quaternion
$$q=\cos(\phi)+\frac{u}{\|u\|}\sin(\phi).
$$
Then, the product $qPq^{-1}$ is the result of rotating $P$ about the vector $u$ through $2\phi$.
\end{theorem}
\subsection{Spherical coordinates}
We now introduce spherical coordinates. These coordinates are used to represent a point in the 3D space. A point in spherical coordinates is written $(\rho, \theta, \phi)$ where
\begin{itemize}
\item $\rho \geq 0$ is the distance from the origin $O$ to $P$;
\item $0\leq \theta < 2\pi$ is the angle formed by the positive $x$ axis and the projection of $P$ in the $xy$ plane;
\item $0\leq \phi \leq \pi$ is the angle formed by the positive $z$ axis and the line segment $OP$.
\end{itemize}
Since pure quaternions are three-dimensional, they can also be used to visualize points in the 3D space using vectors $i,j$ and $k$ as axis. Thus, we can represent pure quaternions in spherical coordinates using the following formulas for a pure quaternion $q=ai+bj+ck$ and $\rho\neq 0$ \cite{calcul}: 
\begin{align*}
\rho=\sqrt{a^2+b^2+c^2}, && \theta\ =\ \arctantwo (b,a), && \phi=\arccos (c/\rho).
\end{align*}
These formulas are used to switch from spherical coordinates to Cartesian coordinates.
Since the tangent function has a period of $\pi$ and that $0 \leq \theta < 2\pi$, we define the function $\arctantwo$ that returns the angle in the right quadrant. This function is already used in many programming softwares.
We define the spherical representation of a pure quaternion as follows by using the above formulas:
$$
q = \rho(i\sin\phi\cos\theta + j\sin\phi\sin\theta + k\cos\phi).
$$
This common definition of spherical coordinates assures unicity of all coordinates except $(0,0,0)$ and the two poles.
To have a unique representation for the quaternion $0$, we fix $\theta=0$ and $\phi=0$. So, the spherical coordinate that represents this quaternion is $(0,0,0)$. For the poles, that are the cases when $\phi=0$ or $\phi=\pi$, we fix $\theta=0$. So, in the case of a quaternion $ck$ where $c>0$, the associated spherical coordinate is $(c,0,0)$. In the case where the quaternion is in the form $ck$ where $c<0$, the associated spherical coordinate is $(|c|,0,\pi)$.
In this way, we have a unique spherical representation for all pure quaternions.

Moreover, we can easily verify that $\|q\|=\|\rho(i\sin\phi\cos\theta +j\sin\phi\sin\theta +k\cos\phi)\|=\rho$. Indeed, we have
\begin{small}
\begin{align*}
\|\rho(i\sin\phi\cos\theta +j\sin\phi\sin\theta +k\cos\phi)\| &=  \sqrt{\rho^2(\sin^2\phi\cos^2\theta+\sin^2\phi\sin^2\theta + \cos^2\phi)}\\
&=  \sqrt{\rho^2}\sqrt{\sin^2\phi(\cos^2\theta+\sin^2\theta) + \cos^2\phi}\\
&= \rho\sqrt{\sin^2\phi + \cos^2\phi}\\
&=  \rho.
\end{align*}
\end{small}%
\section{Spherical product of pure quaternions}\label{sec:spherical}
We know that multiplication of unitary complex numbers represents a rotation in the 2D plane. Thus, we define the spherical product by reproducing this idea in the 3D space (see \cite{Boily}). This product between pure unitary quaternions represents a rotation in the 3D space using spherical coordinates.
\begin{definition}
Consider $q_1$, $q_2 \in$ \textnormal{Im}$\mathbb{H}$ in spherical representation. Then, the spherical product is defined as
\begin{align*}
q_1 \times_s q_2 &:= \rho_1\rho_2 \big( i\sin(\phi_1+\phi_2)\cos(\theta_1+\theta_2)+j\sin(\phi_1+\phi_2)\sin(\theta_1+\theta_2)\\
&+k\cos(\phi_1+\phi_2)\big).
\end{align*}
\end{definition}
\begin{theorem}
$(\textnormal{Im}\mathbb{H},\times_s)$ is a commutative unital magma, in other words, it has the following properties:
\begin{enumerate}
\item \textnormal{Im}$\mathbb{H}$ is closed under $\times_s$,
\item the operation $\times_s$ is commutative,
\item \textnormal{Im}$\mathbb{H}$ has an identity element $e$ with respect to the operation $\times_s$.
\end{enumerate}
\end{theorem}
\begin{proof}
Closure is a direct consequence of the definition. Commutativity is a direct consequence of the commutativity of real numbers for addition and multiplication. Now, we show that the identity element is the quaternion $k$. Let $q_1 \in$ Im$\mathbb{H}$. We have
\begin{align*}
q_1 \times_s k&=\rho_1\cdot 1 \big( i\sin(\phi_1+0)\cos(\theta_1+0)+j\sin(\phi_1+0)\sin(\theta_1+0)\\
&+k\cos(\phi_1+0)\big)\\
&= \rho_1 \big( i\sin(\phi_1)\cos(\theta_1)+j\sin(\phi_1)\sin(\theta_1)+k\cos(\phi_1)\big)\\
&= q_1\\
&= k \times_s q_1.
\end{align*}
Thus, the identity element is $e=k$ since we have $q\times_s k=q= k\times_s q$ for all $q\in$ Im$\mathbb{H}$.
\end{proof}
Since we showed that $($Im$\mathbb{H},\times_s)$ is a commutative unital magma, we now show why it is not a group. To be a group, the operation would need to be associative and all $q \in$ Im$\mathbb{H}$ would need to have an inverse element. The following theorem shows that only quaternions $ck$ where $c\in \mathbb{R}^*$ have an inverse element. 
\begin{theorem}
The only elements in $(\textnormal{Im}\mathbb{H},\times_s)$ that have an inverse are quaternions $ck$ where $c\in \mathbb{R}^*$.
\end{theorem}
\begin{proof}
Consider $q_1 \in$ Im$\mathbb{H}$ and its spherical representation such that $0<\phi_1<\pi$. Therefore, $q_1$ is any pure quaternion except quaternions of the form $ck$ where $c\in \mathbb{R}$. Suppose that $q_1$ has an inverse element $q_2$, then $q_1 \times_s q_2=k$. So, since we need to have $\cos(\phi_1+\phi_2)=1$ to obtain $k$, we need to have $\phi_1+\phi_2=2n\pi$ for $n\in \mathbb{Z}$. For $n=0$, the equation is true if and only if $\phi_1=\phi_2=0$. For $n=1$, the equation is true if and only if $\phi_1=\phi_2=\pi$. The other cases are impossible in the context of the unique spherical representation. Thus, since $0<\phi_1<\pi$, by hypothesis, neither of the equalities can be true. Therefore, $q_1$ doesn't have an inverse. From this reasoning, we can easily see that the inverse of $ck$ is $\frac{k}{c}$ when $c\in\mathbb{R}^*$. For the quaternion $0$, suppose that $0 \times_s q_2 = k$. We need to have $0\cdot \rho_2 =1$. This is impossible, so $0$ doesn't have an inverse. In summary, the only quaternions that have an inverse are in the form $ck$ where $c\in\mathbb{R}^*$ and their inverses are in the form $\frac{k}{c}$.
\end{proof}
Now, we show with a counterexample why the operation is not associative. Consider $q_1 = \frac{\sqrt{2}}{2}j - \frac{\sqrt{2}}{2}k$ and $q_2=j$.
Using carefully the unicity of the spherical representation, we obtain $q_2 \times_s (q_1 \times_s q_1) = -k$ and $(q_2 \times_s q_1) \times_s q_1 = -j$. Thus, $\times_s$ is not associative and $(\textnormal{Im}\mathbb{H},\times_s)$ cannot be a monoid.
\section{The quaternionic Mandelbrot set}
\label{sec:quaternionic}
We present here a generalization of the Mandelbrot set using quaternions. But first, let's recall the definition of the Mandelbrot set. Let $P_c(z) = z^2 +c$ and denote
\[
P_{c}^{(n)}(z)=\underbrace{(P_{c}\circ P_{c}\circ\cdots\circ P_{c})}_{n\text{ times}}(z).
\]
Using the function $P_c$ we can define the standard Mandelbrot set as
$$
\mathcal{M}=\lbrace c\in\ \mathbb{C}\ |\  \lbrace P^{(n)}_c(0)\rbrace_{n\in \mathbb{N}}\ \text{is bounded}\rbrace.
$$
Now, we set $c\in \mathbb{H}$ instead and we have the following generalization.
\begin{definition}
The quaternionic Mandelbrot set is defined as
$$
\mathcal{M}_{\mathbb{H}}=\lbrace c\in\ \mathbb{H}\ |\  \lbrace P^{(n)}_c(0)\rbrace_{n\in \mathbb{N}}\ \text{is bounded}\rbrace.
$$
\end{definition}
Since $\mathcal{M}_{\mathbb{H}}$ is a four-dimensional object, we can only visualize it through 3D projections on three-dimensional subspaces of $\mathbb{H}$. To classify these subspaces, we will use the same notation as these references \cite{BrouilletteRochon, GarantPelletier, RochonParise, VallieresRochon}. This brings us to the following definitions.
\begin{definition}
Let $\im{k},\im{l},\im{m}\in\{1, i, j, k \}$ with $\im{k}\neq\im{l},\,\im{k}\neq\im{m}$ and $\im{l}\neq\im{m}$. The space
\[
\mathbb{H} (\im{k},\im{l},\im{m}) := \vspan\{\im{k},\im{l},\im{m}\}
\]
is the vector subspace of $\mathbb{H}$ consisting of all real finite linear combinations of these three distinct units.
\end{definition}
\begin{definition}
Let $\im{k},\im{l},\im{m}\in\{1,i, j, k \}$ with $\im{k}\neq\im{l},\,\im{k}\neq\im{m}$ and $\im{l}\neq\im{m}$. We define a principal 3D slice of the quaternionic Mandelbrot set $\mathcal{M}_{\mathbb{H}}$ as
\begin{align*}
\mathcal{H}(\im{k},\im{l},\im{m}) &= \{c\in\mathbb{H}(\im{k},\im{l},\im{m}) : \suite\text{ is bounded}\} \\
&= \mathbb{H}(\im{k},\im{l},\im{m})\cap\mathcal{M}_\mathbb{H}.
\end{align*}
\end{definition}
Before we present the 3D slices, we present the following result that can be found in \cite{wang}.
\begin{theorem}
\label{theo:lienmandelbrot}
Consider $q=q_0+q_1i+q_2j+q_3k$ where $\rho=\sqrt{q_1^2+q_2^2+q_3^2}$. We have
$$
q \in \mathcal{M}_{\mathbb{H}} \Leftrightarrow c = q_0 + \rho i \in \mathcal{M}.
$$
\end{theorem}
We have four possible combinations of $\im{k}, \im{l}, \im{m} \in \{1,i, j, k \}$ and the corresponding 3D slices are $\mathcal{H}(1,i,j)$, $\mathcal{H}(1,i,k)$, $\mathcal{H}(1,j,k)$ and $\mathcal{H}(i,j,k)$.
\begin{theorem}
\label{theo:quatslice}
The 3D slices $\mathcal{H}(1,i,j)$, $\mathcal{H}(1,i,k)$ and $\mathcal{H}(1,j,k)$ are rotations of the classical Mandelbrot set $\mathcal{M}$ around the real axis. 
\end{theorem}
\begin{proof}
We need to show that for every quaternion $q= q_0 +q_1\im{m}+q_2\im{l}$ we have
$$
q \in \mathcal{M}_{\mathbb{H}} \Leftrightarrow q_0 + i \sqrt{q_1^2+q_2^2} \in \mathcal{M}
$$
where $\im{k}, \im{l}\in \{ i,j,k\}$ with $\im{k} \neq \im{l}$. We show it directly using Theorem \ref{theo:lienmandelbrot} in the particular case where $q\in \mathbb{H}$ is in the form $q= q_0 +q_1\im{k}+q_2\im{l}$.
\end{proof}
\begin{figure}[htp]
\centering
\includegraphics[width=.3\textwidth]{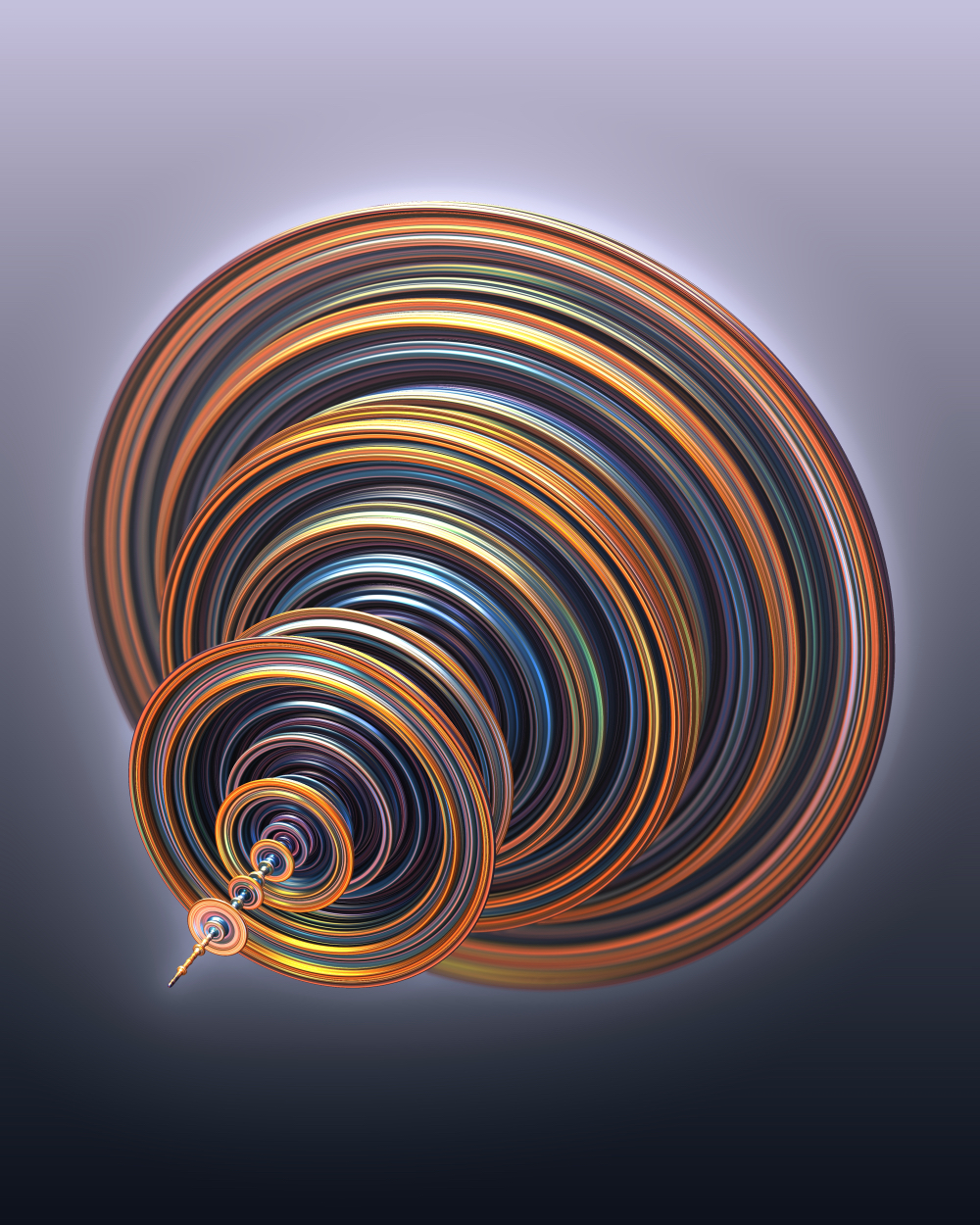}\hfill
\includegraphics[width=.3\textwidth]{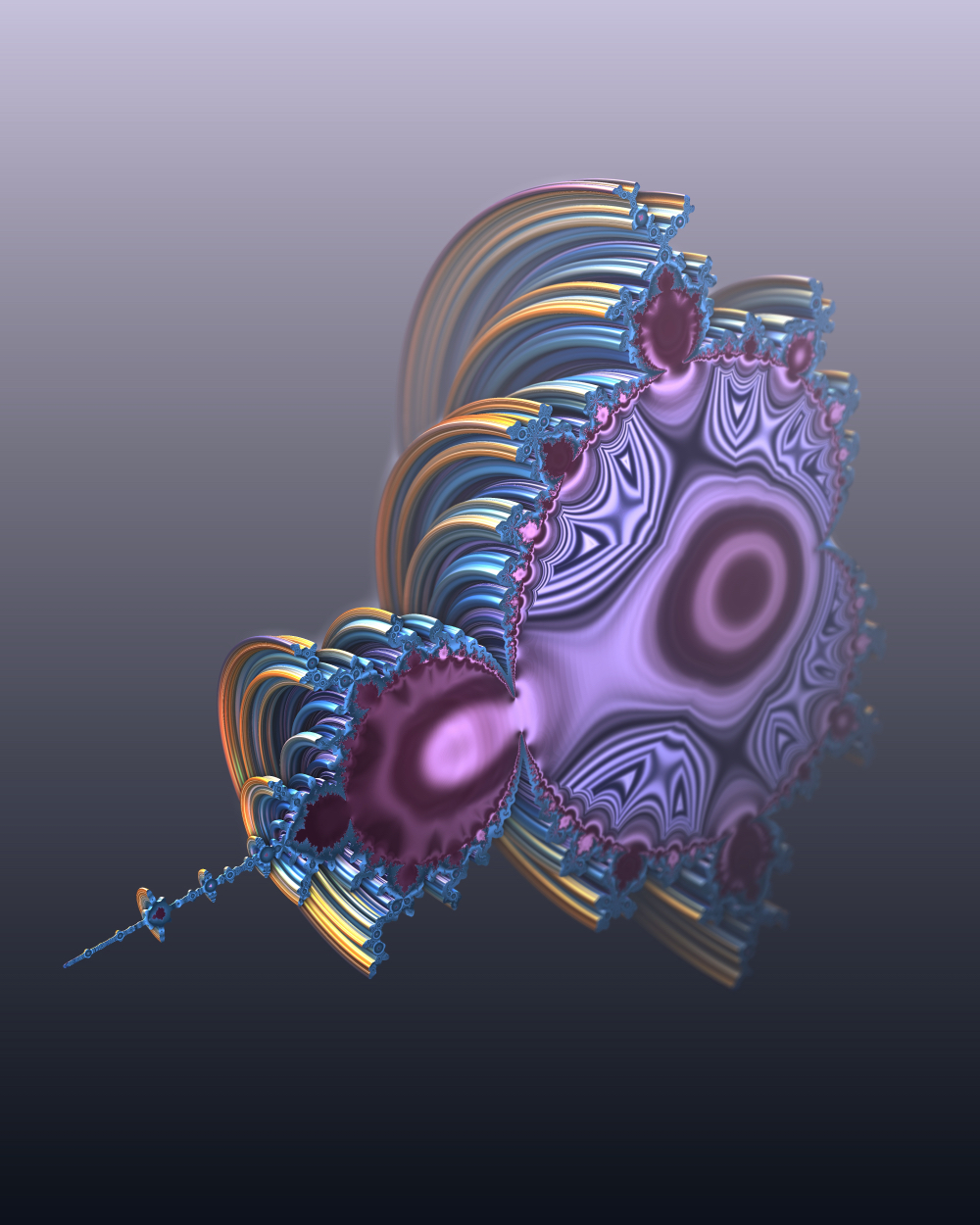}\hfill
\includegraphics[width=.3\textwidth]{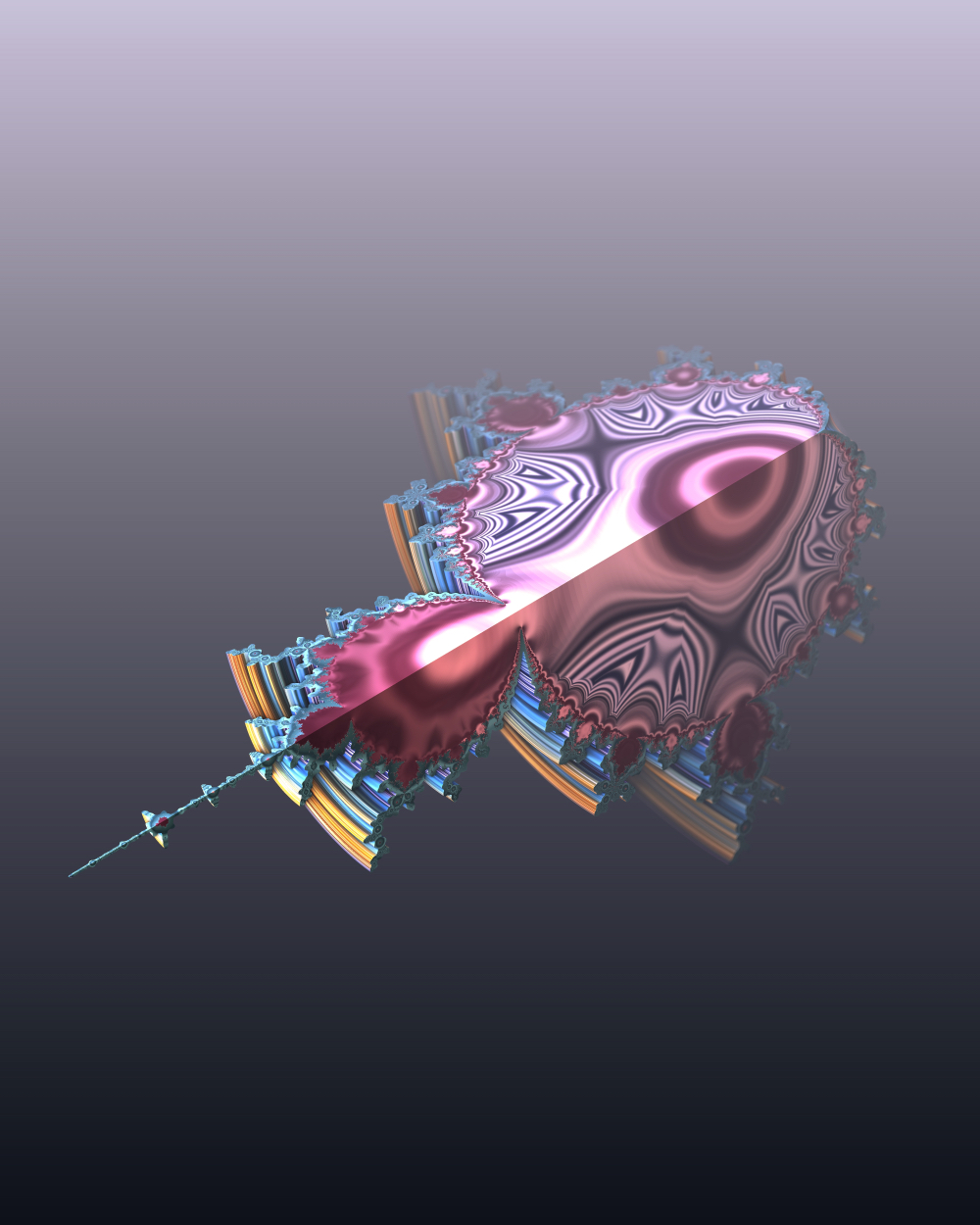}
\caption{\label{fig:mandelquat}Representation of the 3D slice $\mathcal{H}(1,i,j)$.The internal dynamics is provided by an escape time algorithm from the Mandelbulb 3D software.}
\end{figure}
Figure \ref{fig:mandelquat} shows the slice $\mathcal{H}(1,i,j)$. We can see that it is a rotation of the standard Mandelbrot set around the real axis. Now, we present the 3D slice $\mathcal{H}(i,j,k)$. We found that the result in \cite{wang} is incorrect. The slice is not a sphere but is contained in a closed ball. 
\begin{figure}[htp]
\centering
\includegraphics[scale=0.07]{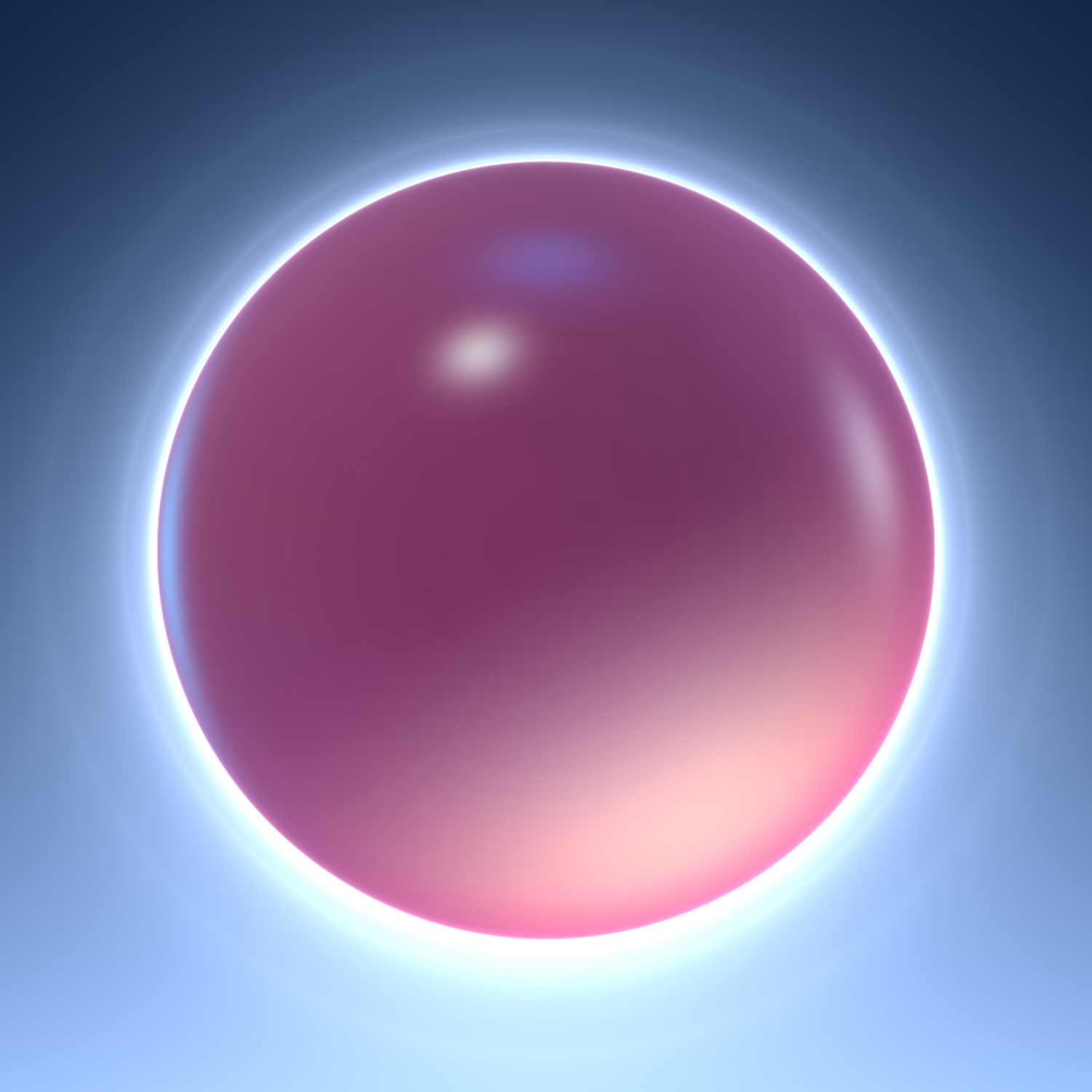}\hspace{10mm}
\includegraphics[scale=0.07]{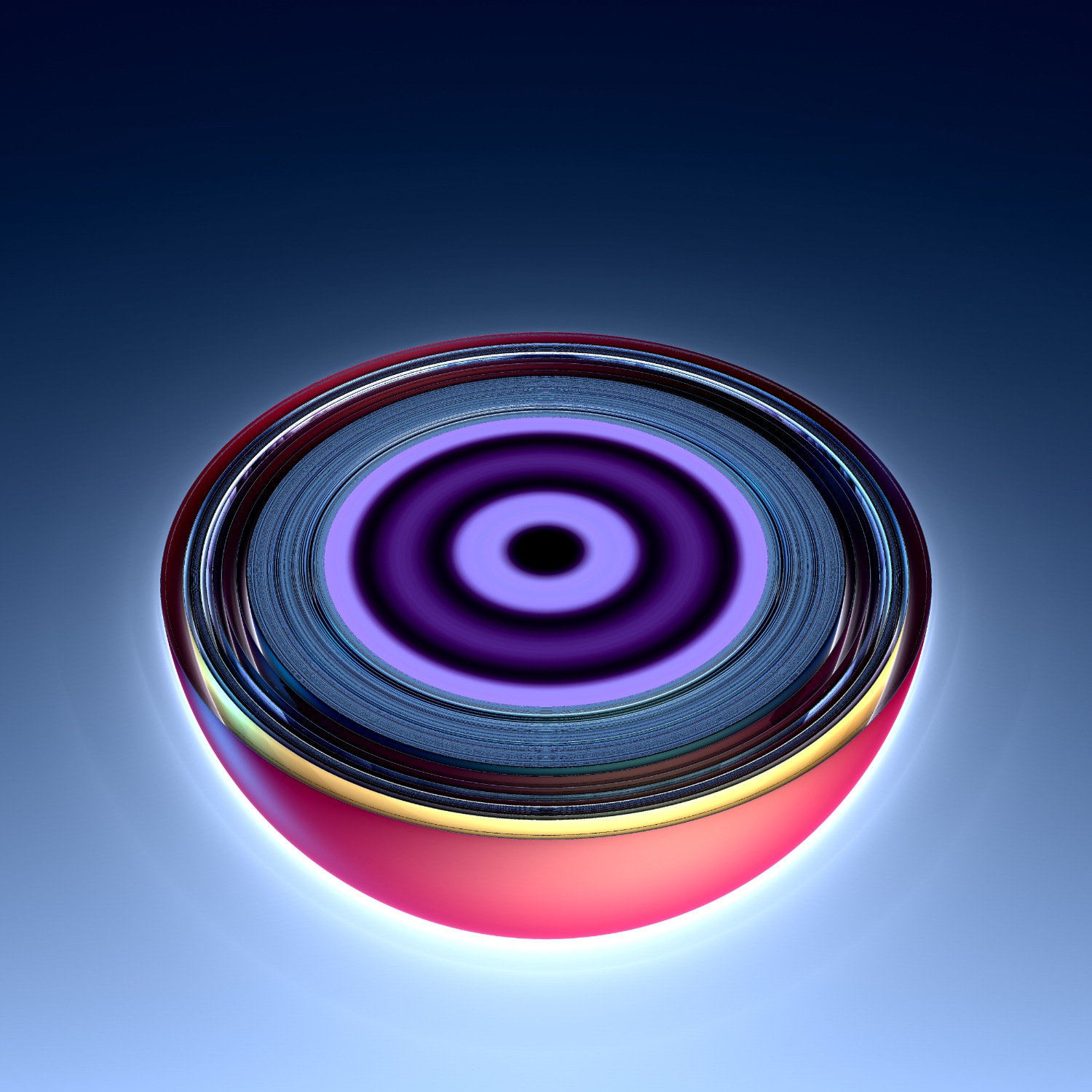}\hfill
\caption{\label{fig:mandelquatijk}Representation of the Metasphere, the 3D slice $\mathcal{H}(i,j,k)$.}
\end{figure}
\begin{figure}[htp]
\centering
\includegraphics[scale=0.12]{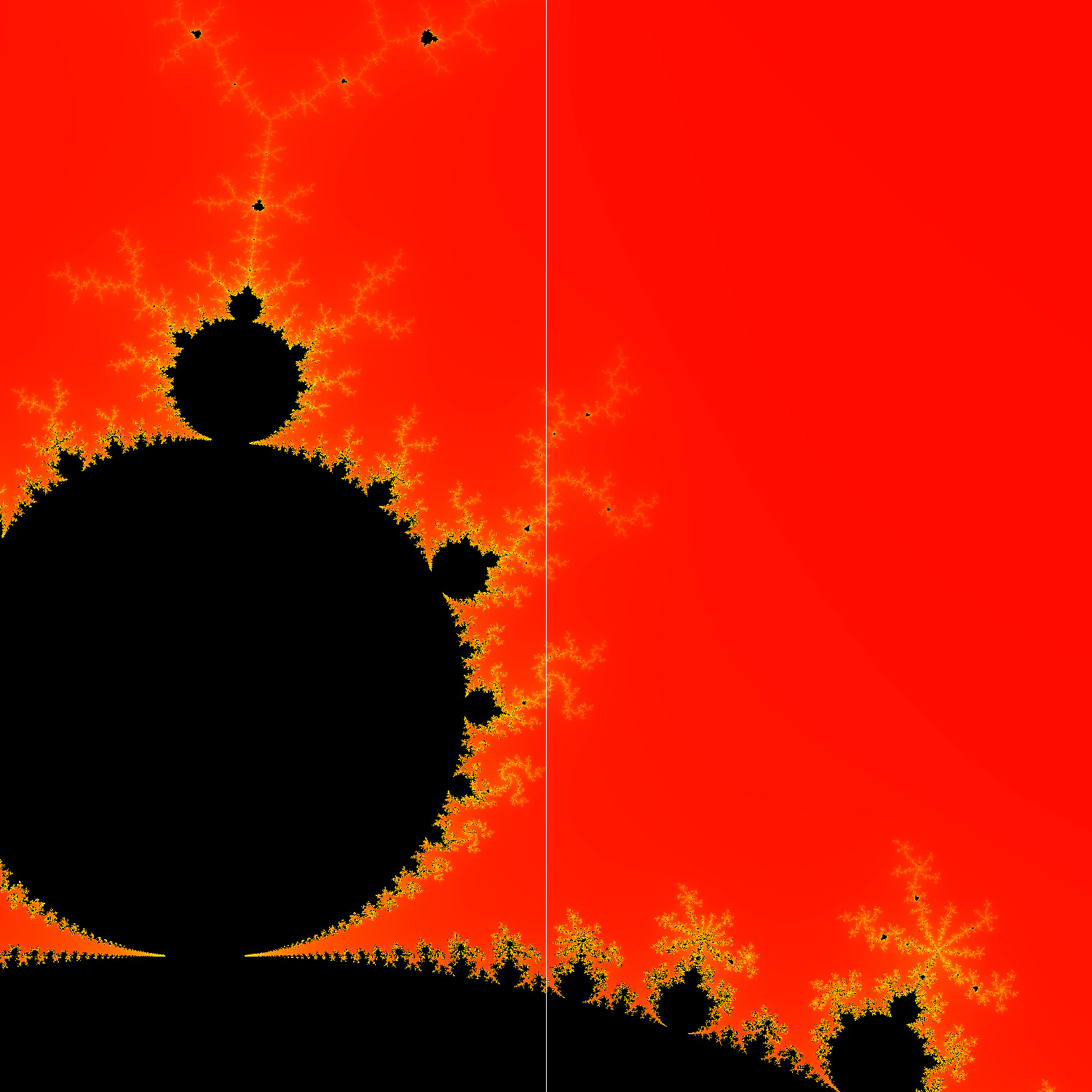}
\caption{\label{fig:axesmandel}Zoom of the intersection between imaginary axis (in white) and the Mandelbrot set $\mathcal{M}$.}
\end{figure}

\begin{theorem}
The 3D slice $\mathcal{H}(i,j,k)$ is contained in a closed ball.
\end{theorem}
\begin{proof}
For all pure quaternions $q$, using Theorem \ref{theo:lienmandelbrot}, we have
$$
q \in \mathcal{M}_{\mathbb{H}} \Leftrightarrow \|q\|i \in \mathcal{M}.
$$
So, $q\in \mathcal{M}_{\mathbb{H}}$ if and only if $\|q\|i$ is in the intersection between $\mathcal{M}$ and the complex imaginary axis. Let $R=$max$\lbrace xi \in \mathcal{M} \rbrace$. This maximum exists since $\mathcal{M}$ is closed. Thus, if $\|q\| = R$, then $\|q\|i \in \mathcal{M}$ and so $q \in \mathcal{M}_{\mathbb{H}}$. However, if  $\|q\| > R$, then $\|q\|i \notin \mathcal{M}$ and so $q \notin \mathcal{M}_{\mathbb{H}}$.
We conclude that the 3D slice $\mathcal{H}(i,j,k)$ is contained in a closed ball of radius $R$.
\end{proof}
Figure \ref{fig:mandelquatijk} shows the slice $\mathcal{H}(i,j,k)$, we call it the \textit{Metasphere}. We can see that the slice is contained in a 3D ball. However, if we cut this ball, we can see several nested balls. In fact, there is a sphere for each intersection point between the Mandelbrot set $\mathcal{M}$ and the imaginary axis. Since $\mathcal{M}\cap i\mathbb{R}$ is well known to be disconnected, as we can see in Figure \ref{fig:axesmandel}, we obtain empty spaces between the nested balls.

In summary, the quaternionic Mandelbrot set has two principal slices which are $\mathcal{H}(1,i,j)$ and $\mathcal{H}(i,j,k)$. These slices are respectively the rotation of the standard Mandelbrot set around the real axis and the Metasphere. We note that this result contrasts with the three principal 3D slices (Tetrabrot, Arrowheadbrot and Mousebrot) of the Mandelbrot set generalized in the bicomplex space \cite{BrouilletteRochon, VallieresRochon}.

\section{Generalized Mandelbrot set in true 3D}\label{sec:mandelbrot}
\subsection{The spherical Mandelbrot set}
In this section, we generalize the Mandelbrot set using the spherical product we defined in section \ref{sec:spherical}. First of all, we need to define what is a spherical polynomial using the spherical power.
%
\begin{definition}
Let $q$ be a pure quaternion in spherical representation. The spherical power is noted:
$$
sp_n(q):=\rho^n\left(i\sin(n\phi)\cos(n\theta)+j\sin(n\phi)\sin(n\theta)+k\cos(n\phi)\right).
$$
\end{definition}
\begin{definition}
A spherical polynomial is defined as
$$P(q):= a_n \times_s (sp_n(q)) + a_{n-1} \times_s (sp_{n-1}(q)) + ... + a_1\times_s q + a_0 $$ where $a_n \neq 0$, $a_0,...,a_n \in$ \textnormal{Im}$\mathbb{H}$, $ q \in$ \textnormal{Im}$\mathbb{H}$ and $n \in \mathbb{N}$ is the degree of the spherical polynomial $P(q)$.
\end{definition}
Obviously, the spherical power $sp_2(q)$ is equivalent to the spherical product $q \times_s q$. Now, we define the spherical Mandelbrot set. Let's consider the spherical polynomial $Q_c(q)= sp_2(q) + c$ where $q, c \in$ Im$\mathbb{H}$.
\begin{definition}
The spherical Mandelbrot set is defined as
$$\mathcal{M}_{\textnormal{Im}\mathbb{H}}:=\lbrace c\in\  \textnormal{Im}\mathbb{H}\ |\  \lbrace Q^{(n)}_c(0)\rbrace_{n\in \mathbb{N}}\ \text{is bounded}\rbrace .$$
\end{definition}

%
%
\begin{theorem}
\label{mnormspher}
Let $n\in\mathbb{N}$. If $q$ is a pure quaternion then
$$
\| sp_{n}(q) \| = \|q\|^n.
$$
\end{theorem}
\begin{proof}
Consider $q$ a pure quaternion. By definition of the spherical representation $\|q\|=\rho$. Moreover, we have
$$
sp_n(q) = \rho^n(i\sin n\phi\cos n\theta + j\sin n\phi\sin n\theta + k\cos n\phi).
$$
We can conclude that $\| sp_n(q) \| = \rho^n$, since 
$$
\| i\sin n\phi\cos n\theta + j\sin n\phi\sin n\theta + k\cos n\phi\|=1
$$
Therefore, we have $\| sp_n(q)\| = \|q\|^n  = \rho^n$.
\end{proof}
The following lemma is analogous to a lemma in complex numbers and will be used to prove an important theorem to generate images of the spherical Mandelbrot set. 
\begin{lemma}
\label{prealableborne}
Consider the spherical polynomial $Q_c(q) = sp_2(q) + c$ where $q, c \in$ \textnormal{Im}$\mathbb{H}$ with $\|c\|> 2$, then $\|Q^{(n)}_c(0)\| \geq \|c\|(\|c\|-1)^{n-1}$ with $n\geq 1$.
\end{lemma}
\begin{proof}
We prove it by mathematical induction. For $n=1$, we have
$$
\|Q_c(0)\|=\|c\|=\|c\|(\|c\|-1)^{1-1},
$$
so the statement is true for $n=1$.
Suppose the statement is true for $n=k$, then
$$
\|Q_c^{(k)}(0)\| \geq \|c\|(\|c\|-1)^{k-1}.
$$
We now show that the statement is true for $n=k+1$. We have
\begin{align*}
\|Q^{(k+1)}_c(0)\|&=\|Q_c(Q^{(k)}c(0))\|\\
&=\|sp_2(Q^{(k)}_c(0)) + c\|.
\end{align*}
Using the triangle inequality, the definition of the norm and Theorem \ref{mnormspher}, we have
\begin{align*}
\|sp_2(Q^{(k)}_c(0)) + c\| &\geq  \| sp_2(Q^{(k)}_c(0))  \| - \| c\|\\
&= \| Q^{(k)}_c(0) \|^2 - \|c\|.
\end{align*}
Using the induction hypothesis, we have
$$
\| Q^{(k)}_c(0) \|^2 - \|c\| \geq (\|c\|(\|c\|-1)^{k-1})^2 - \|c\|.
$$
Moreover, since $\|c\| > 2$, we have $(\|c\|-1)^{k-1} \geq 1$. Thus, we have
\begin{align*}
(\|c\|(\|c\|-1)^{k-1})^2 - \|c\| &\geq  \|c\|^2(\|c\|-1)^{k-1} - \|c\|\\
&\geq  \|c\|^2(\|c\|-1)^{k-1} - \|c\|(\|c\|-1)^{k-1}\\
&= \|c\|(\|c\|-1)^{k-1}(\|c\|-1)\\
&= \|c\|(\|c\|-1)^{k}.
\end{align*}
We obtain $\|Q_c^{(k+1)}(0)\| \geq \|c\|(\|c\|-1)^{k}$. Thus, the statement is true for $n=k+1$. We showed by induction that the statement is true for all $n\geq 1$.
\end{proof}
Now, we can show another property of the spherical Mandelbrot set which is analogous to a property of the classical Mandelbrot set.
\begin{theorem}
\label{norm2}
For all pure quaternions $c$ in the spherical Mandelbrot set, we have $\|c\| \leq 2$.
\end{theorem}
\begin{proof}
Suppose $\|c\| > 2$. Then, from Lemma  \ref{prealableborne}, we have $\|Q^{(n)}_c(0)\| \geq \|c\|(\|c\|-1)^{n-1}$ with $n\geq 1$. Thus, we have $\|Q^{(n)}_c(0)\| \rightarrow \infty$, since $\|c\| > 2$. Indeed, $(\|c\|-1)^{n-1} \rightarrow \infty$, when $\|c\| > 2$. So, since $\|Q^{(n)}_c(0)\|$ is not bounded, then $c\notin \mathcal{M}_{\text{Im}\mathbb{H}}$. Thus, by contrapositive, we obtain that if $c\in \mathcal{M}_{\text{Im}\mathbb{H}}$, then $\|c\| \leq 2$.
\end{proof}

This theorem allows to conclude that the spherical Mandelbrot set is totally included in a ball of radius 2 centered at 0. 
\begin{lemma}
\label{lemma}
Let $\delta >0$. If $\|c\| \leq 2$ and $\|Q_c^{m}(0)\|= 2+\delta > 2$ for $m \geq 1$, then $\| Q_c^{(m+n)}(0) \| \geq 2+4^n\delta$ where $n\geq 1$.
\end{lemma}
\begin{proof}
We prove it by mathematical induction. For $n=1$, we have
$$
\| Q_c^{(m+1)}(0) \| =\| Q_c(Q_c^{(m)}(0)) \| = \|sp_2(Q^{(m)}_c(0)) + c\|.
$$
Using the triangle inequality, the definition of the norm and Theorem \ref{mnormspher}, we have
\begin{align*}
\| sp_2(Q^{(m)}_c(0))  + c\| &\geq  \| sp_2(Q^{(m)}_c(0))   \| - \| c\|\\
&= \| Q^{(m)}_c(0) \|^2 - \|c\| \\
&= (2+\delta)^2 - \|c\|.
\end{align*} 
But, since $\|c\| \leq 2 $, we have $(2+\delta)^2-\|c\| \geq (2+\delta)^2 -2 \geq 2+4\delta$. Thus, the statement is true for $n=1$. Suppose the statement is true for $n=k$, then $\| Q_c^{(m+k)}(0) \| \geq 2+4^k\delta$. Similarly to the case where $n=1$, we have
\begin{align*}
\| Q^{(m+k+1)}_c(0) \| &=  \| Q_c(Q^{(m+k)}_c(0)) \|\\
&=\|sp_2(Q^{(m+k)}_c(0))  + c\| \\
&\geq  \| sp_2(Q^{(m+k)}_c(0))  \| - \| c\|\\
&= \| Q^{(m+k)}_c(0) \|^2 - \|c\|.
\end{align*} 
Using the induction hypothesis and the fact that $\|c\| \leq 2$, we have
$$
\| Q^{(m+k)}_c(0) \|^2 - \|c\| \geq (2+4^k\delta)^2 - 2 = 2 + 4\cdot 4^k\delta + 4^{2k}\delta^2 \geq 2 + 4^{k+1}\delta.
$$
Thus, the statement is true for $n=k+1$. We showed by induction that the statement is true for all $n\geq 1$.
\end{proof}

\begin{figure}[htp]
\centering
	\begin{subfigure}{0.3\textwidth}
	\centering
	\includegraphics[width=3.5cm]
	{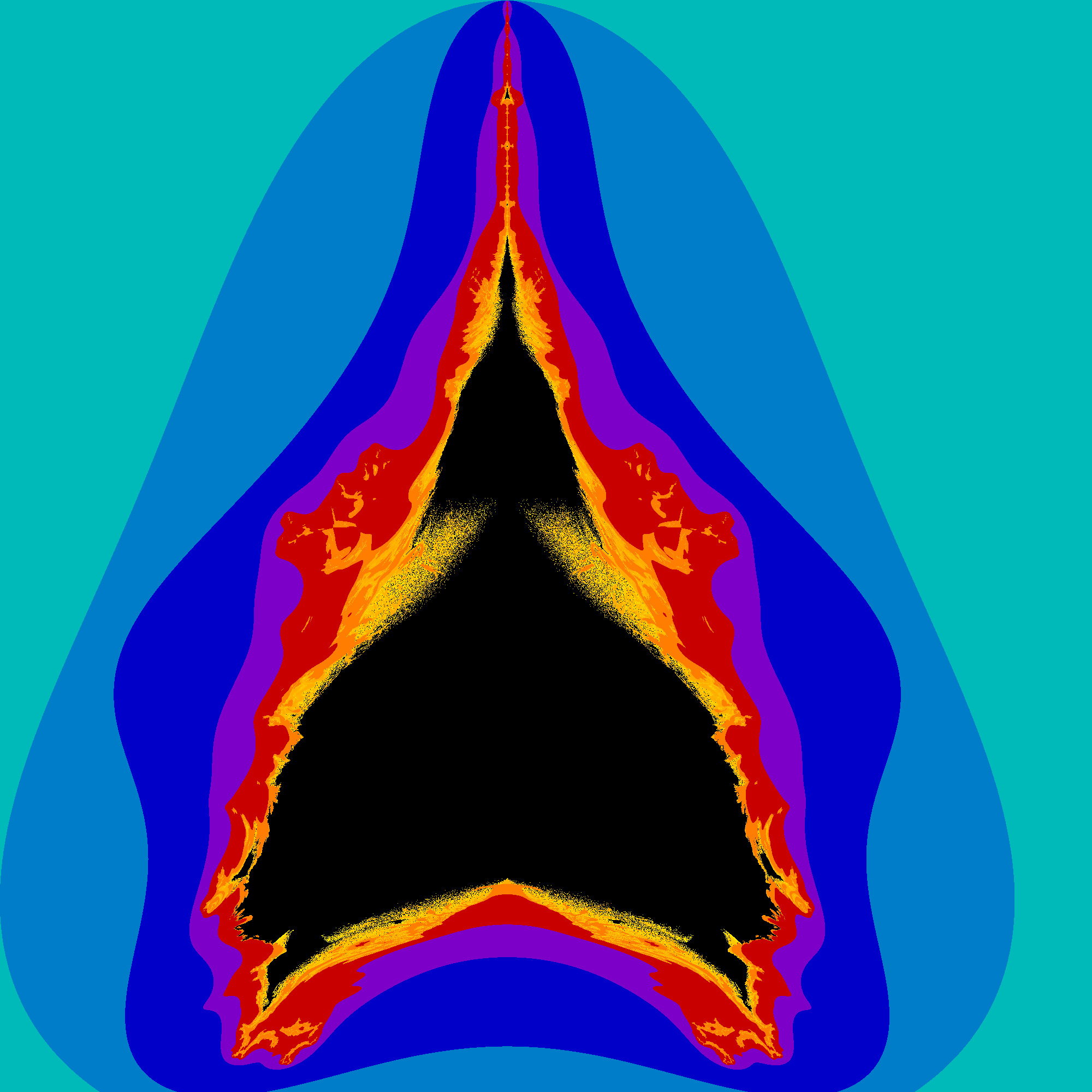}
	\caption{$\mathcal{M}_{\text{Im}\mathbb{H}}$ cut for $x=0$}
	\label{fig:x0}
	\end{subfigure}
\hfill
	\begin{subfigure}{0.3\textwidth}
	\centering
	\includegraphics[width=3.5cm] 
	{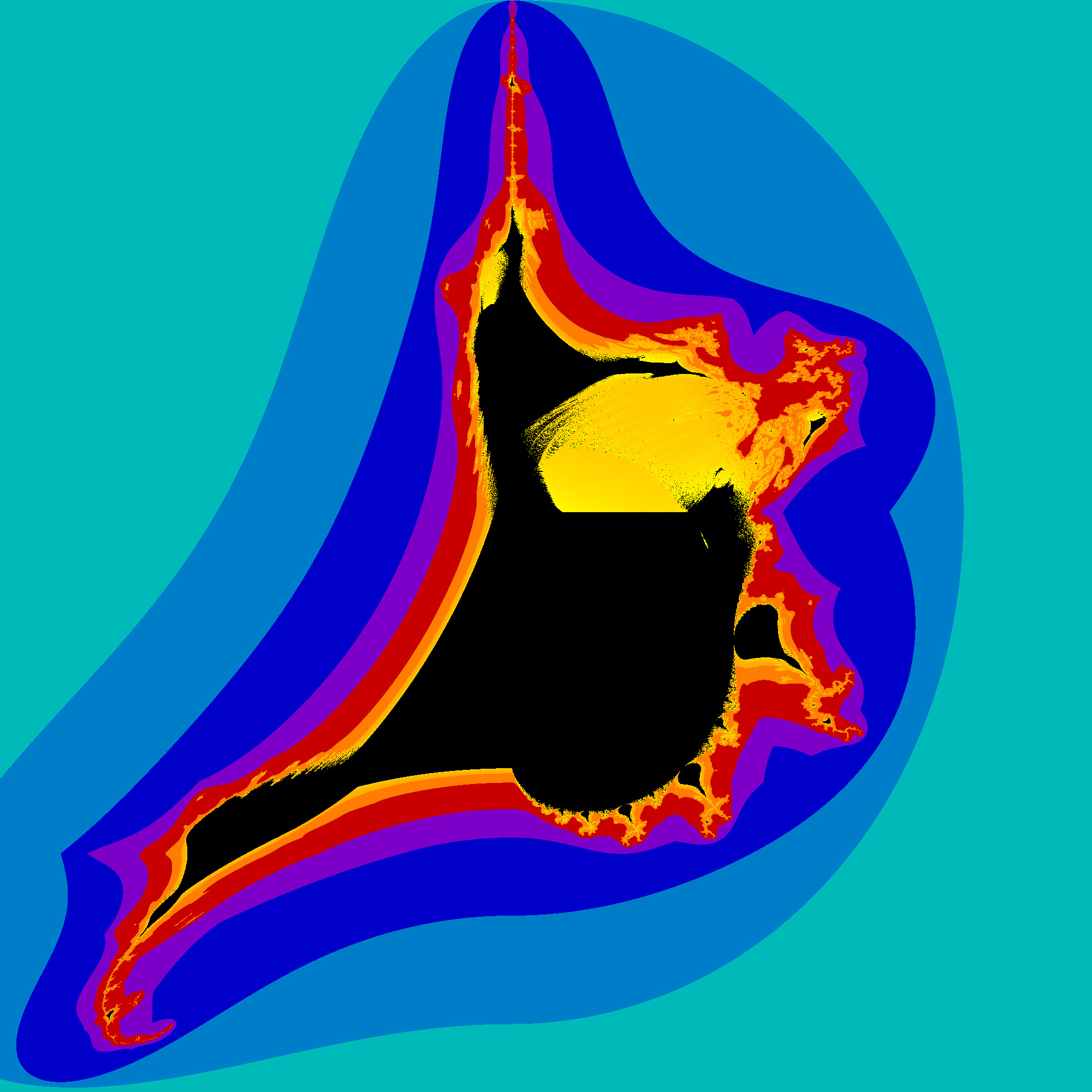}
	\caption{$\mathcal{M}_{\text{Im}\mathbb{H}}$ cut for $y=0$}
	\label{fig:y0}
	\end{subfigure}
\hfill
	\begin{subfigure}{0.3\textwidth}
	\centering
	\includegraphics[width=3.5cm]
	{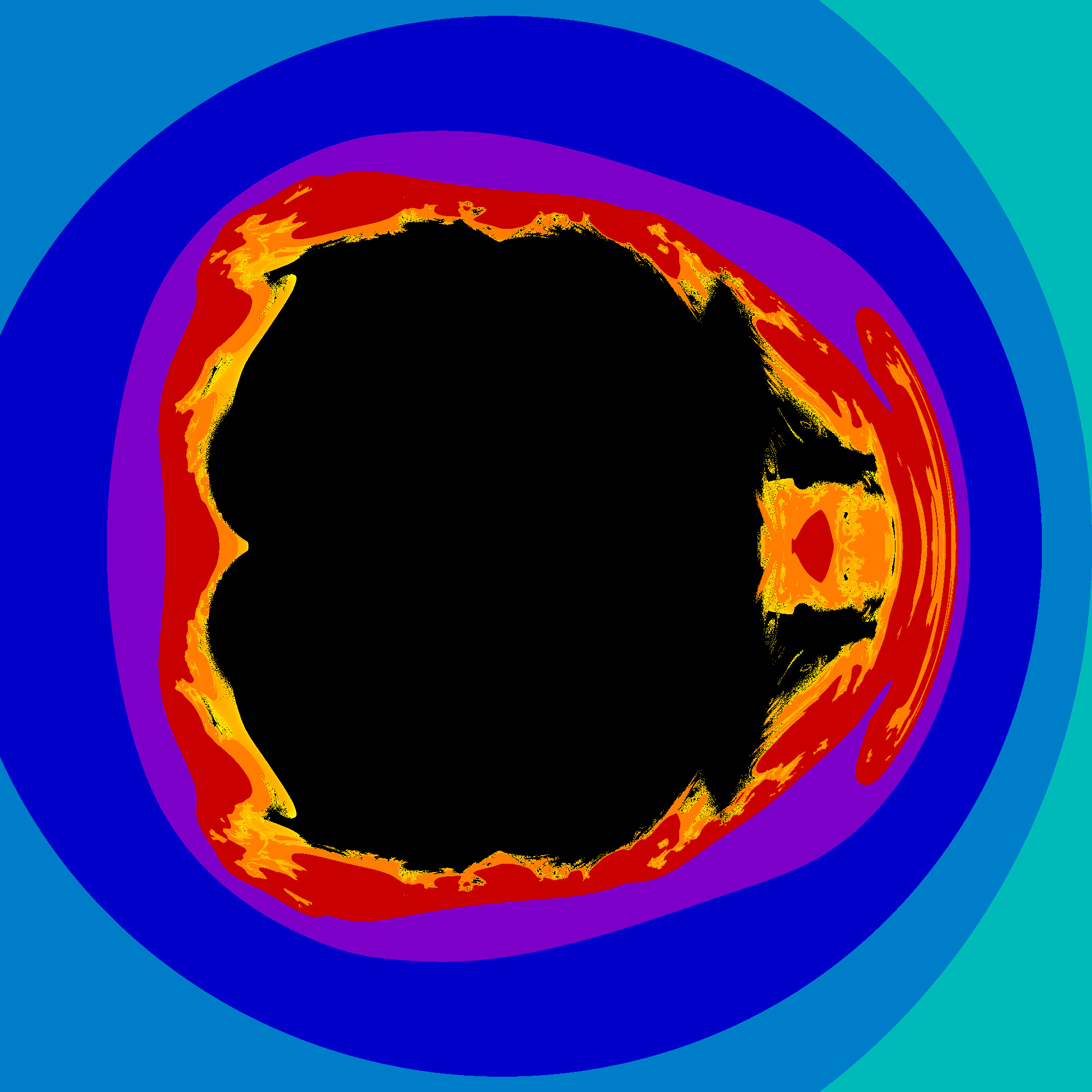}
	\caption{$\mathcal{M}_{\text{Im}\mathbb{H}}$ cut for $z=0$}
	\label{fig:z0}
	\end{subfigure}
\\ \vspace*{5pt}
	\begin{subfigure}{0.3\textwidth}
	\centering
	\includegraphics[width=3.5cm]
	{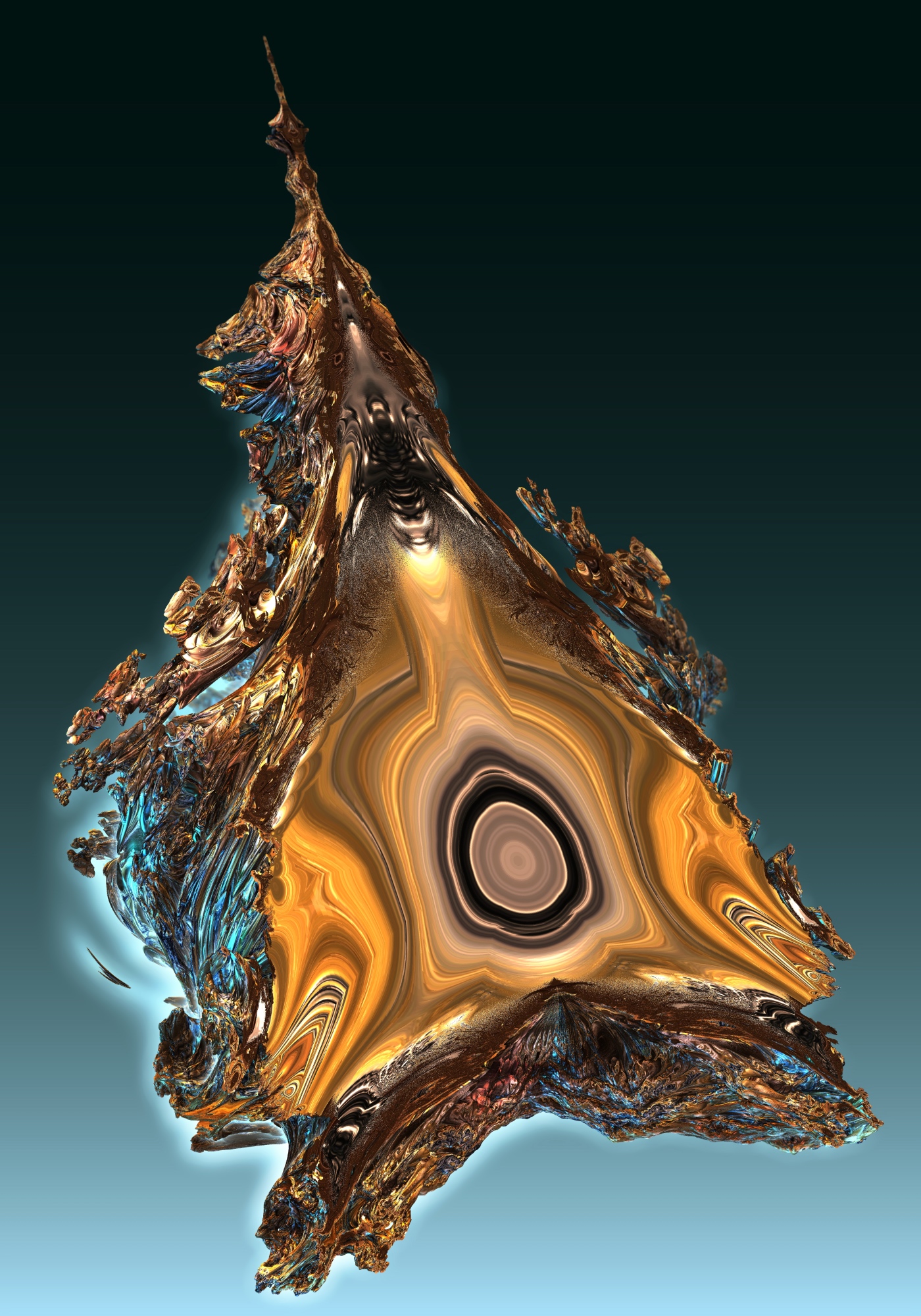}
	\caption{IQ Bulb cut for $x=0$}
	\label{fig:IQBulbx0}
	\end{subfigure}
\hfill
	\begin{subfigure}{0.3\textwidth}
	\centering
	\includegraphics[width=3.5cm]
	{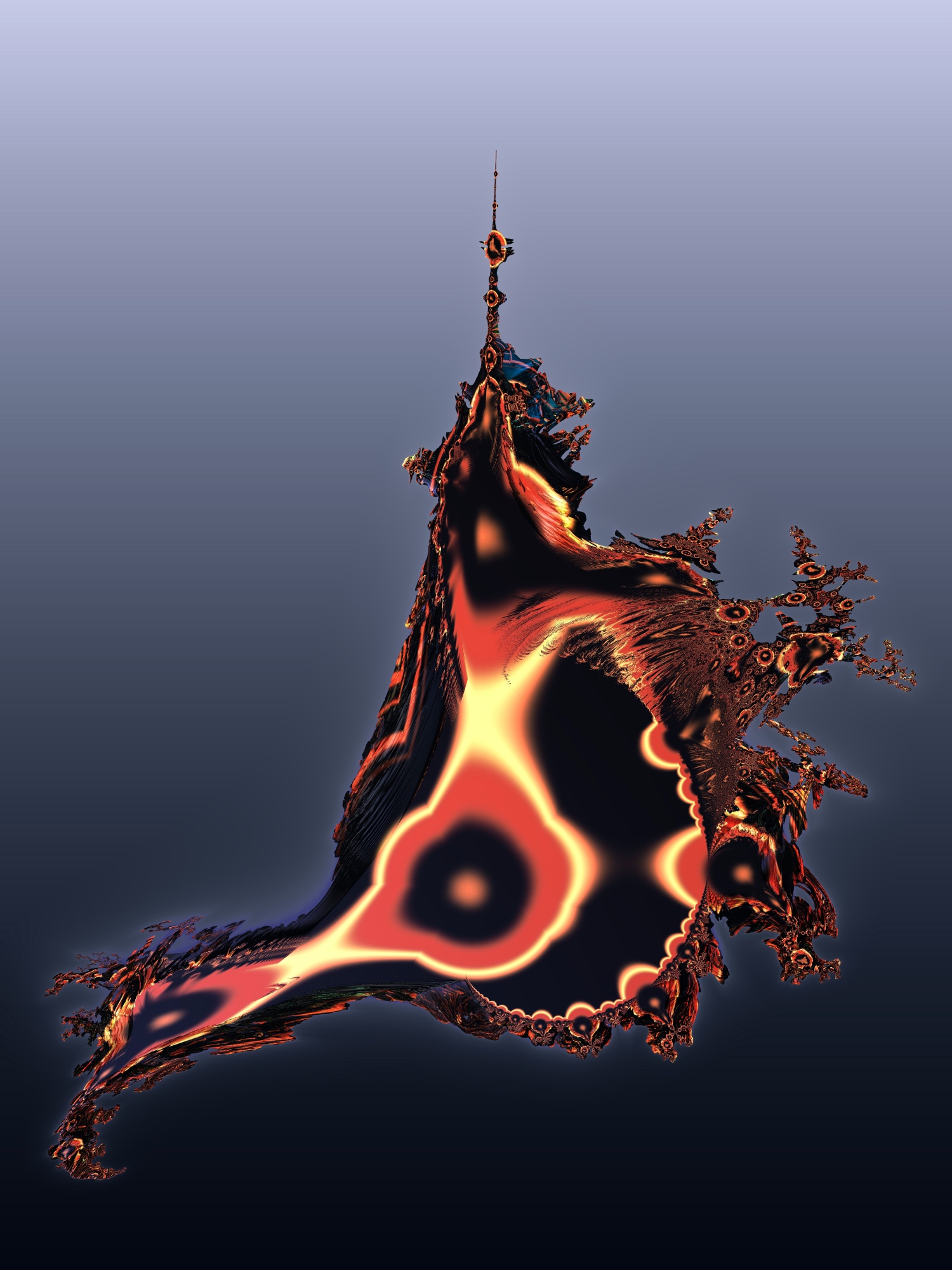}
	\caption{IQ Bulb cut for $y=0$}
	\label{fig:IQBulby0}
	\end{subfigure}
\hfill
	\begin{subfigure}{0.3\textwidth}
	\centering
	\includegraphics[width=3.5cm]
	{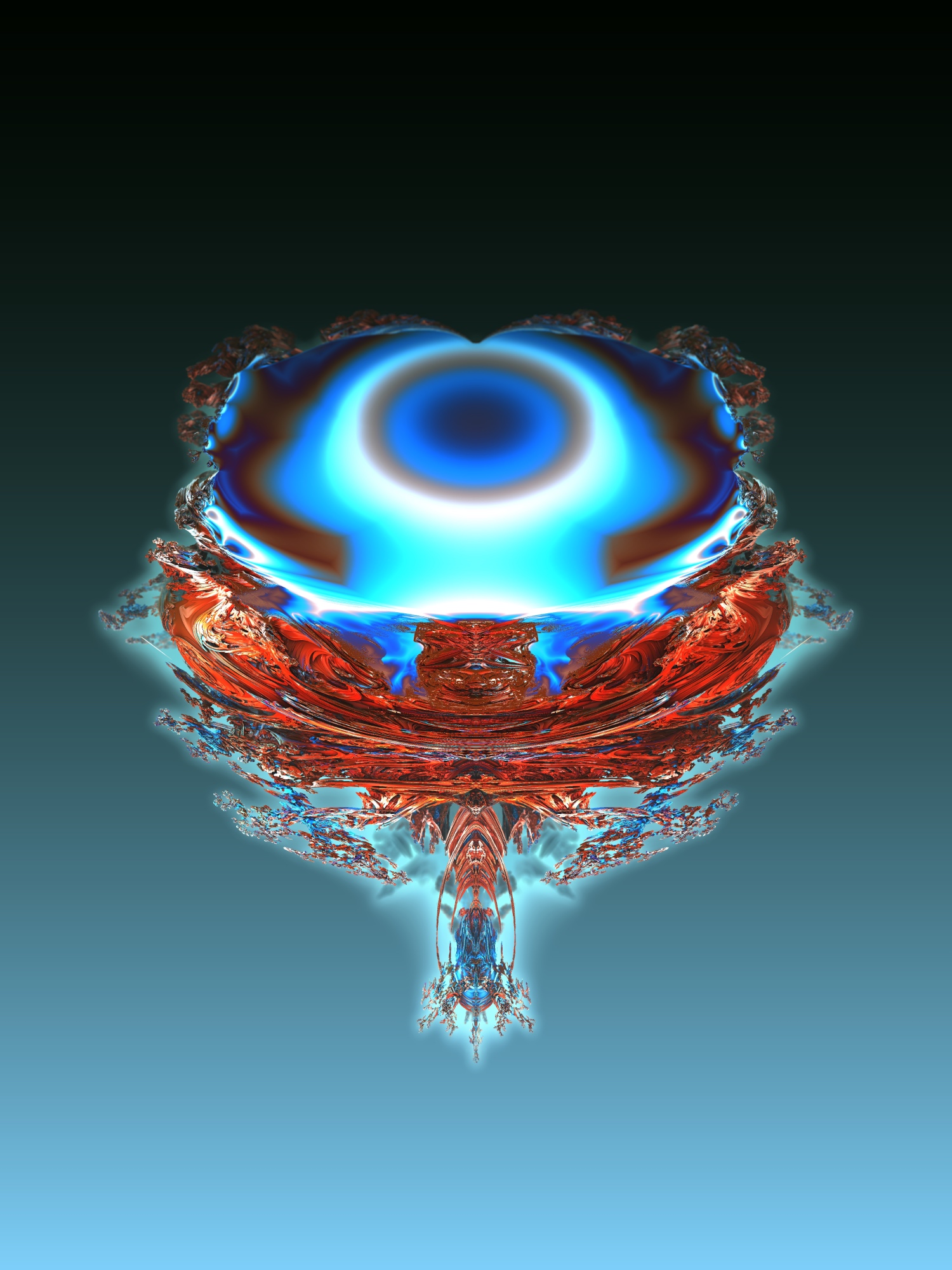}
	\caption{IQ Bulb cut for $z=0$}
	\label{fig:IQBulbz0}
	\end{subfigure}
\caption{Comparison between 2D cuts of the IQ Bulb and the $\mathcal{M}_{\text{Im}\mathbb{H}}$.}
\label{fig:coupes}
\end{figure}

\begin{theorem}
\label{theo:borne}
A pure quaternion $c$ is in the spherical Mandelbrot set if and only if $\|Q_c^{(n)}(0)\|\leq 2$ for all $n$.
\end{theorem}
\begin{proof}
The fact that if $\|Q_c^{(n)}(0)\|\leq 2$ for all $n$, then $c \in \mathcal{M}_{\textnormal{Im}\mathbb{H}}$ is a direct consequence of the sherical Mandelbrot set definition. We now show that if $c \in \mathcal{M}_{\textnormal{Im}\mathbb{H}}$, then $\|Q_c^{(n)}(0)\|\leq 2$ for all $n$. Let $c \in \mathcal{M}_{\textnormal{Im}\mathbb{H}}$. From Theorem \ref{norm2}, we have $\|c\| \leq 2$. Suppose there exists an $n\in\mathbb{N}$ such that $\|Q_c^{(n)}(0)\|= 2+\delta$ for some $\delta >0$. Then, from Lemma \ref{lemma}, $\| Q_c^{(m+n)}(0) \| \geq 2+4^n\delta$. Thus, $\| Q_c^{(m+n)}(0) \| \rightarrow \infty$ which means that $\| Q_c^{(n)}(0) \| \rightarrow \infty$ and that $ c \notin \mathcal{M}_{\textnormal{Im}\mathbb{H}}$. We have a contradiction. Therefore, we have $\|Q_c^{(n)}(0)\| \leq 2$ for all $n$. 
\end{proof}
Theorem \ref{theo:borne} is useful to visualize the set. Indeed, to generate images using the computer we need to verify if the iterates are bounded by $2$. Figure \ref{fig:coupes} shows cuts of the spherical Mandelbrot set. We generated these images using Python. In black, we can see the points that are in the set. The points around the set have been colored according to the iteration from which $\| Q_c^{(n)}(0)\| > 2$. This method is the same that is used to generate the Mandelbrot set and its colors.

\subsection{The Mandelbulb}
In 2007, Daniel White and Paul Nylander used iterations of spherical coordinates to generate images using a ray tracing method inspired by other methods used in \cite{brouillettepariserochon, dang, rochonmartineau}. They showed images for the power 8. This fractal is known as the Mandelbulb \cite{Barrallo, nylander, white}. Then, in 2009, Inigo Quilez improved the algorithm making it possible to have a good representation of the power 2, the IQ Bulb \cite{CNRS, quilez}. Figure \ref{fig:IQBulb3D} shows the IQ Bulb (also called the \textit{Mandeldart}). More specifically, Figure \ref{fig:coupes} shows slices of the spherical Mandelbrot set in comparison with slices from the IQ Bulb. We see that the slices of the IQ Bulb and the spherical Mandelbrot set are visually the same. Images of the IQ Bulb have been generated with Mandelbulb 3D and the Figure \ref{fig:IQBulbZoom3D} provides some special zooms in this Mandelbulb power 2.\\

\begin{figure}[htp]
\centering
\includegraphics[scale=0.07]{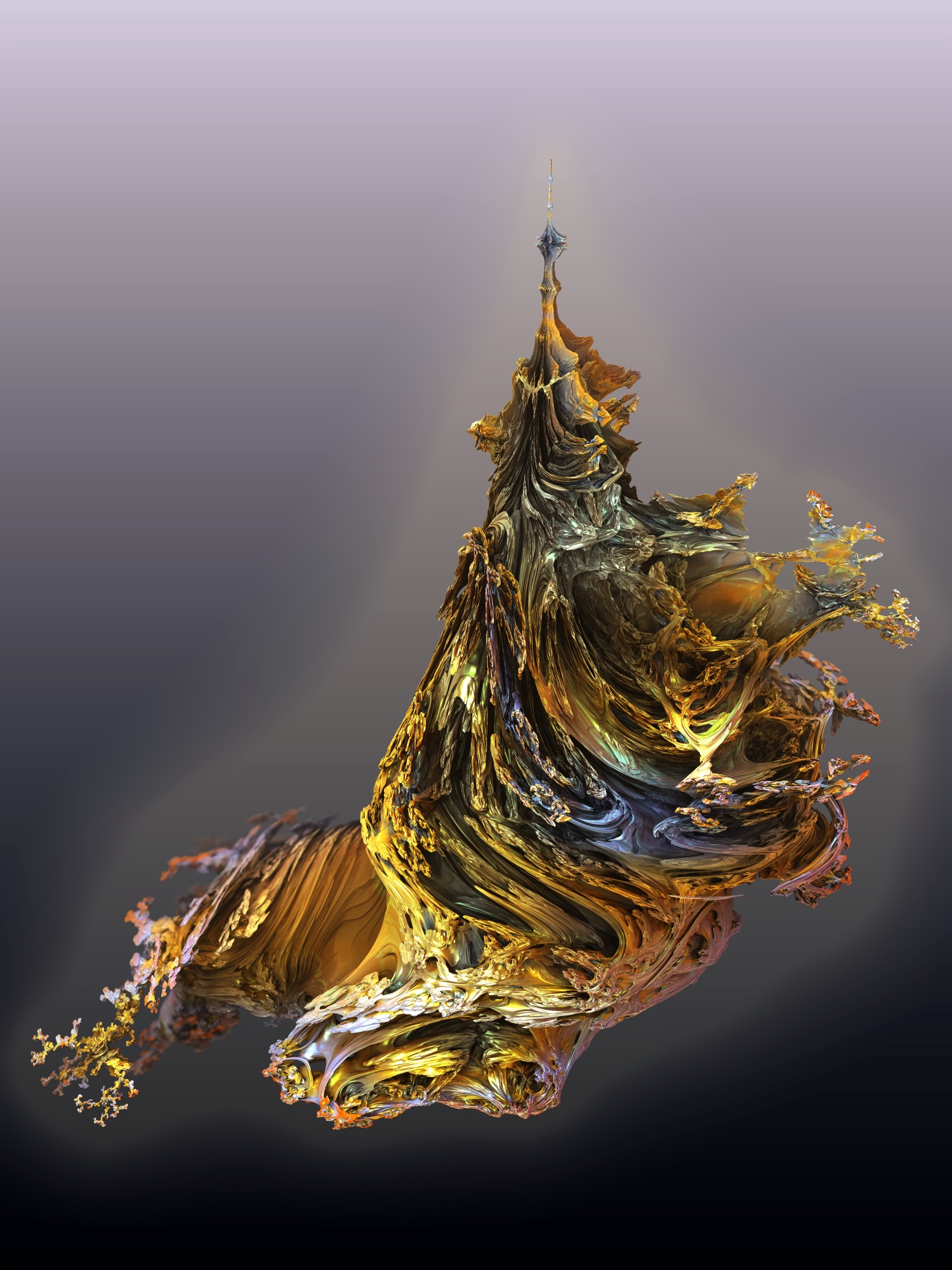}\hspace{10mm}
\includegraphics[scale=0.07]{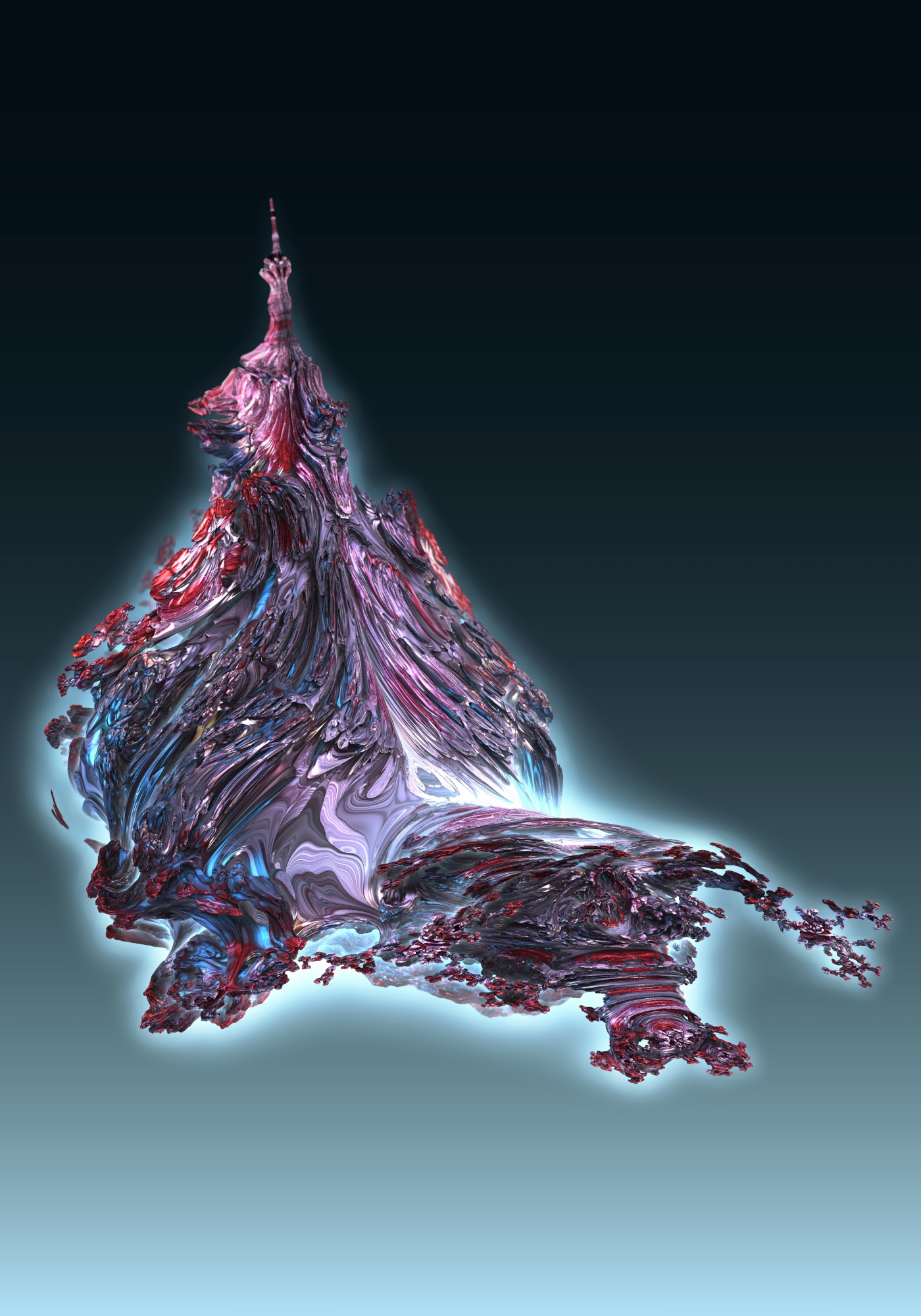}\hfill
\caption{\label{fig:IQBulb3D}Representation of the IQ Bulb generated with Mandelbulb 3D.}
\end{figure}

\begin{figure}[htp]
\centering
\includegraphics[scale=0.07]{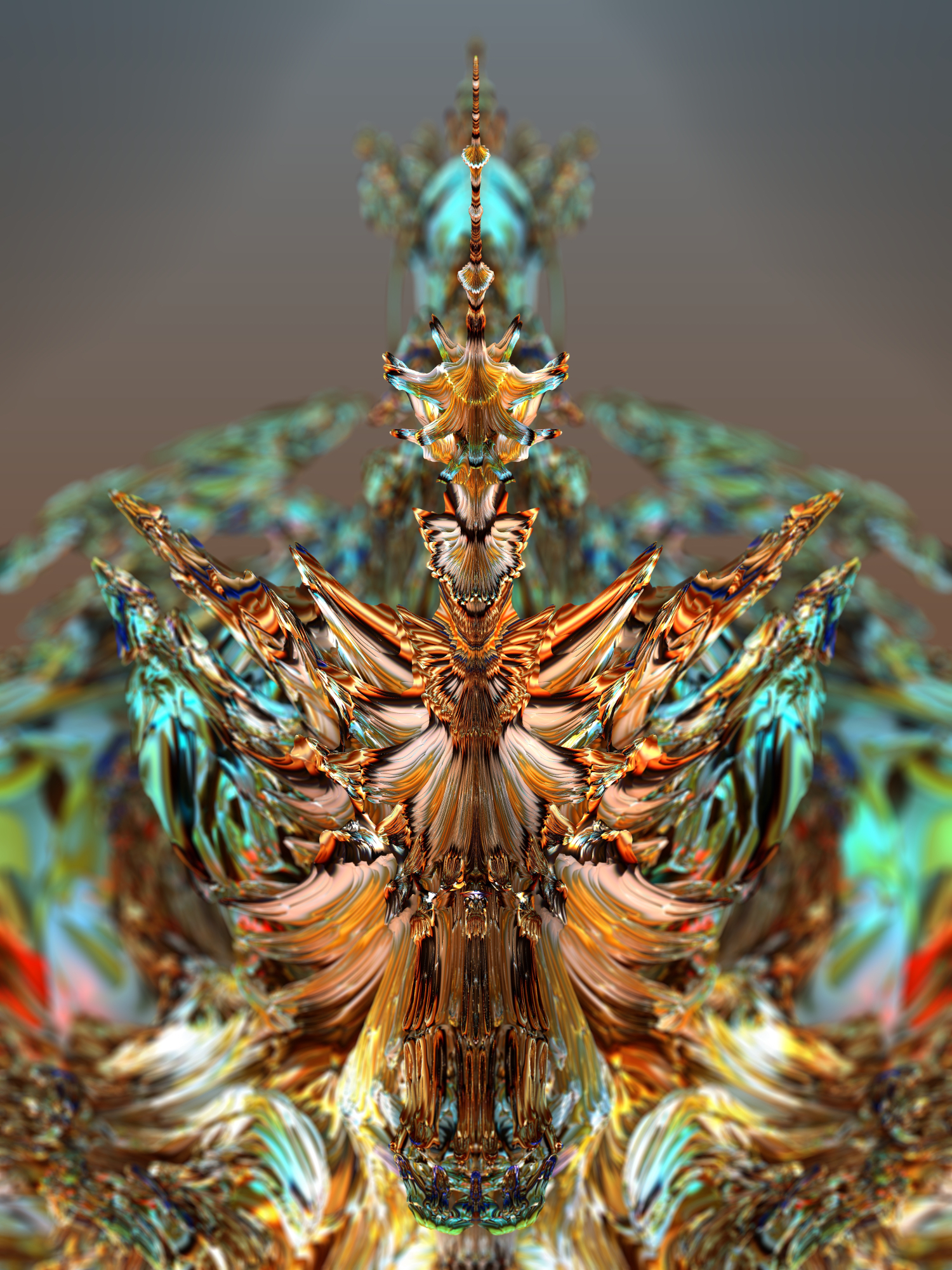}\hspace{10mm}
\includegraphics[scale=0.07]{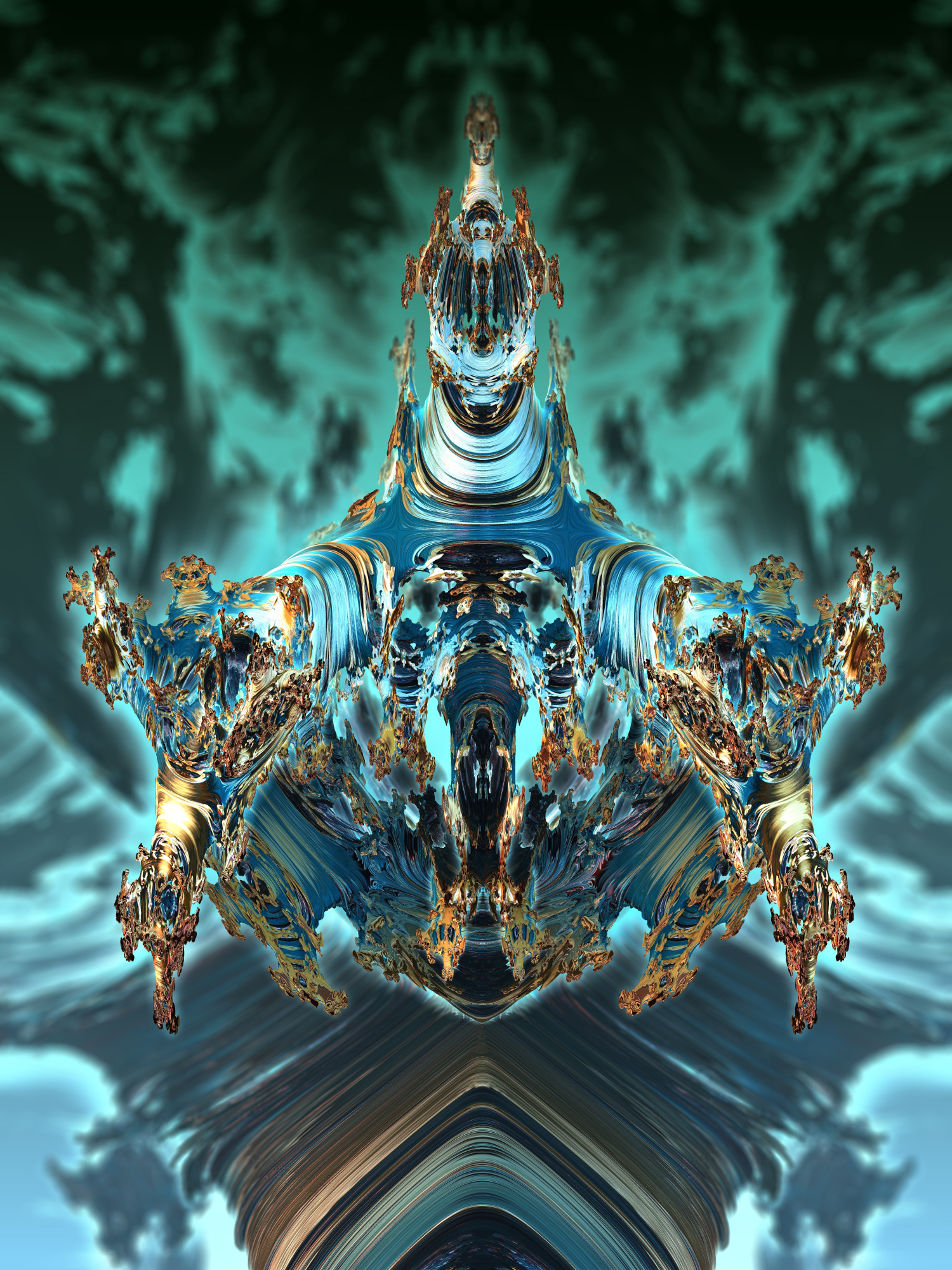}\hfill
\caption{\label{fig:IQBulbZoom3D}Deep zooms on the Mandeldart.}
\end{figure}

Now, we present a new set when the spherical polynomial is of degree $m$. 
\begin{definition}
Let $m\in\mathbb{N}$. The Mandelbrot set $\spher^{m}$ is defined as
$$
\spher^{m}=\lbrace c\in\ \textnormal{Im}\mathbb{H}\ |\  \lbrace Q^{(n)}_{m,c}(0)\rbrace_{n\in \mathbb{N}}\ \text{is bounded}\rbrace
$$
where $Q^{(n)}_{m,c}(q) = sp_m(q) +c$ is a spherical polynomial.
\end{definition}
The following theorems are analogous to theorems showed by Parisé and Rochon in \cite{parise, RochonParise} in the case of Multibrots.
\begin{lemma}
\label{mprealableborne}
Consider the spherical polynomial $Q_{m,c}(q) = sp_m(q) + c$ where $q, c \in$ \textnormal{Im}$\mathbb{H}$ with $m\geq 2$ and $\|c\|^{m-1}> 2$, then $\|Q^{(n)}_{m,c}(0)\| \geq \|c\|(\|c\|^{m-1}-1)^{n-1}$ with $n\geq 1$.
\end{lemma}
\begin{proof}
We prove it by mathematical induction. For $n=1$, we have
$$
\|Q_{m,c}(0)\|=\|c\|=\|c\|(\|c\|^{m-1}-1)^{1-1},
$$
so the statement is true for $n=1$.
Suppose the statement is true for $n=k$, then
$$
\|Q_{m,c}^{(k)}(0)\| \geq \|c\|(\|c\|^{m-1}-1)^{k-1}.
$$
We now show that the statement is true for $n=k+1$. We have
\begin{align*}
\|Q^{(k+1)}_{m,c}(0)\|&=\|Q_{m,c}(Q^{(k)}_{m,c}(0))\|\\
&=\|sp_m(Q^{(k)}_{m,c}(0))  + c\|.
\end{align*}
Using the triangle inequality, the definition of the norm and Theorem \ref{mnormspher}, we have
\begin{align*}
\|sp_m(Q^{(k)}_{m,c}(0)) + c\| &\geq  \| sp_m(Q^{(k)}_{m,c}(0))  \| - \| c\|\\
&= \| Q^{(k)}_{m,c}(0) \|^m - \|c\|.
\end{align*}
Using the induction hypothesis, we have
$$
\| Q^{(k)}_{m,c}(0) \|^m - \|c\| \geq (\|c\|(\|c\|^{m-1}-1)^{k-1})^m - \|c\|.
$$
Moreover, since $\|c\|^{m-1} > 2$, we have $(\|c\|^{m-1}-1)^{k-1} \geq 1$. Thus, we have
\begin{align*}
(\|c\|(\|c\|^{m-1}-1)^{k-1})^m - \|c\| &\geq  \|c\|^m(\|c\|^{m-1}-1)^{k-1} - \|c\|\\
&\geq  \|c\|^m(\|c\|^{m-1}-1)^{k-1} - \|c\|(\|c\|^{m-1}-1)^{k-1}\\
&= \|c\|(\|c\|^{m-1}-1)^{k-1}(\|c\|^{m-1}-1)\\
&= \|c\|(\|c\|^{m-1}-1)^{k}.
\end{align*}
We obtain $\|Q_{m,c}^{(k+1)}(0)\| \geq \|c\|(\|c\|^{m-1}-1)^{k}$. Thus, the statement is true for $n=k+1$. We showed by induction that the statement is true for all $n\geq 1$.
\end{proof}
\begin{theorem}
\label{mnorm}
For all pure quaternions $c$ in Mandelbrot set $\spher^{m}$, we have $\|c\| \leq 2^{1/(m-1)}$.
\end{theorem}
\begin{proof}
Suppose $\|c\| > 2^{1/(m-1)}$ which is equivalent to $\|c\|^{m-1} > 2$. Then, from Lemma \ref{mprealableborne}, we have $\|Q^{(n)}_{m,c}(0)\| \geq \|c\|(\|c\|^{m-1}-1)^{n-1}$ with $n\geq 1$. Thus, we have $\|Q^{(n)}_{m,c}(0)\| \rightarrow \infty$, since $\|c\|^{m-1} -1 > 1$. So, since $\|Q^{(n)}_{m,c}(0)\|$ is not bounded, then $c\notin \spher^m$. Thus, by contrapositive, we obtain that if $c\in \spher^m$, then $\|c\| \leq 2^{1/(m-1)}$.
\end{proof}

This theorem allows to conclude that the Mandelbrot set $\spher^m$ is totally included in a ball of radius $2^{1/(m-1)}$ centered at 0. 
\begin{lemma}
\label{mlemma}
If $\|c\| \leq 2^{1/(m-1)}$ and $\|Q_{m,c}^{(n)}(0)\|= 2^{1/(m-1)}+\delta > 2^{1/(m-1)}$ for $n \geq 1$, then $\| Q_{m,c}^{(n+k)}(0) \| \geq 2^{1/(m-1)}+(2m)^k\delta$ where $k\geq 1$.
\end{lemma}
\begin{proof}
We prove it by mathematical induction. For $k=1$, we have
$$
\| Q_{m,c}^{(n+1)}(0) \| =\| Q_{m,c}(Q_{m,c}^{(n)}(0)) \| = \|sp_m(Q^{(n)}_{m,c}(0)) + c\|.
$$
Using the triangle inequality, the definition of the norm and Theorem \ref{mnormspher}, we have
\begin{align*}
\|sp_m(Q^{(n)}_{m,c}(0)) + c\| &\geq  \| sp_m(Q^{(n)}_{m,c}(0))  \| - \| c\|\\
&= \| Q^{(n)}_{m,c}(0) \|^m - \|c\| \\
&= (2^{1/(m-1)}+\delta)^m - \|c\|\\
&= \sum^m_{i=0} \binom{m}{i} \left(2^{1/(m-1)}\right)^{m-i}(\delta)^i - \| c\| \\
&\geq  \sum^m_{i=0} \binom{m}{i} \left(2^{1/(m-1)}\right)^{m-i}(\delta)^i - 2^{1/(m-1)}.
\end{align*} 
But, since $2^{1/(m-1)} > 1 $ and $\delta > 0$, all terms in the sum are positive. So, 
\begin{align*}
\sum^m_{i=0} \binom{m}{i} \left(2^{1/(m-1)}\right)^{m-i}(\delta)^i - 2^{1/(m-1)} &\geq  2^{m/(m-1)} + \binom{m}{1}2\delta - 2^{1/(m-1)}\\
&\geq  2^{1/(m-1)}(2^m-1) + 2m\delta\\
&\geq  2^{1/(m-1)} + 2m\delta.
\end{align*}
Thus, $\| Q_{m,c}^{n+1}(0)\| \geq 2^{1/(m-1)} + 2m\delta$. The statement is true for $k=1$. Suppose the statement is true for $k=j$, then $\| Q_{m,c}^{(n+j)}(0) \| \geq 2^{1/(m-1)}+(2m)^j\delta$. Similarly to the case where $n=1$, we have
\begin{align*}
\| Q^{(n+j+1)}_{m,c}(0) \| &= \| Q_{m,c}(Q^{(n+j)}_{m,c}(0)) \|\\
&=\|sp_m(Q^{(n+j)}_{m,c}(0))  + c\| \\
&\geq  \|sp_m(Q^{(n+j)}_{m,c}(0)) \| - \| c\|\\
&= \| Q^{(n+j)}_{m,c}(0) \|^m - \|c\|.
\end{align*} 
Using the induction hypothesis and the fact that $\|c\| \leq 2^{1/(m-1)}$, we have
\begin{align*}
\| Q^{(n+j)}_{m,c}(0) \|^m - \|c\| &\geq  (2^{1/(m-1)}+(2m)^j\delta)^m - \|c\| \\
&\geq  2^{m/(m-1)} + \binom{m}{1}2(2m)^j\delta - 2^{1/(m-1)} \\
&\geq  2^{1/(m-1)} + (2m)^{j+1}\delta.
\end{align*}
Thus, the statement is true for $k=j+1$. We showed by induction that the statement is true for all $k\geq 1$.
\end{proof}

\begin{figure}[ht]
\centering
	\begin{subfigure}{0.3\textwidth}
	\centering
	\includegraphics[width=3.5cm]	{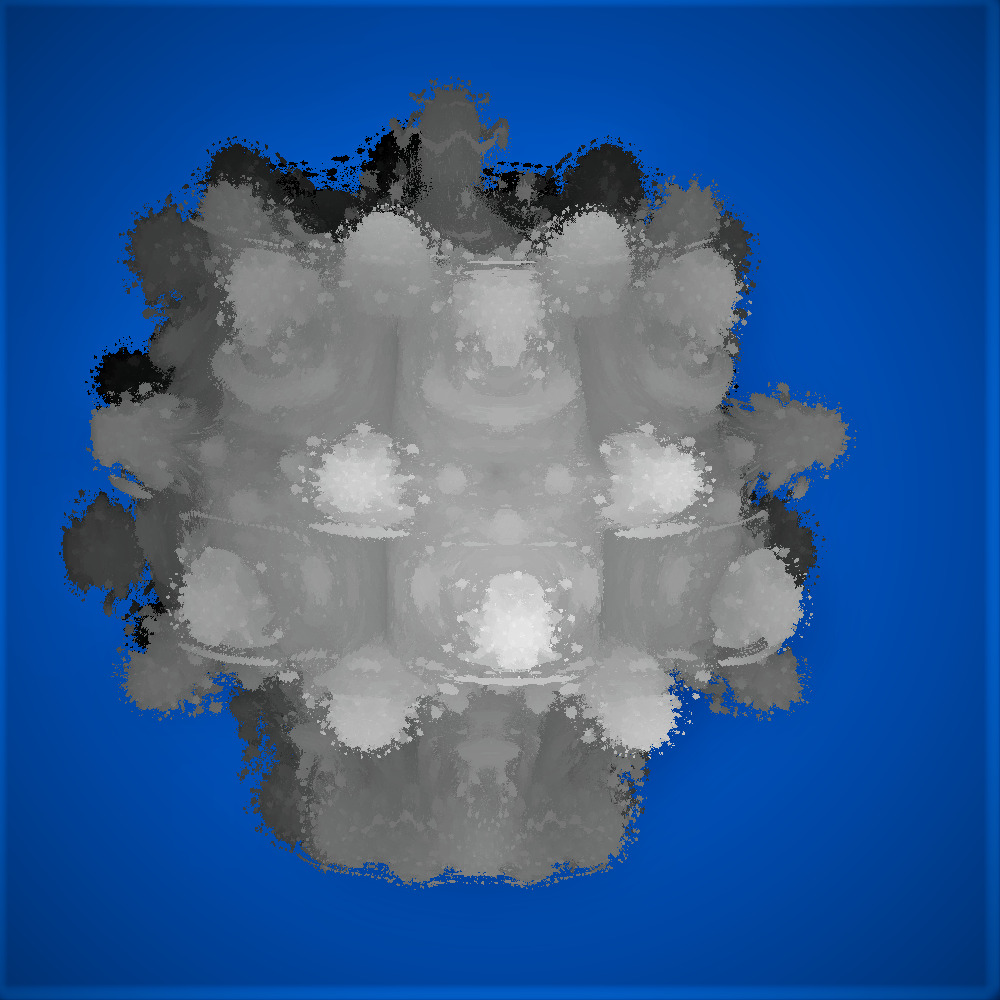}
	\caption{$\spher^{8}$ generated in Python.}
	\end{subfigure}
\hfill
\begin{subfigure}{0.3\textwidth}
	\centering
	\includegraphics[width=3.5cm]
{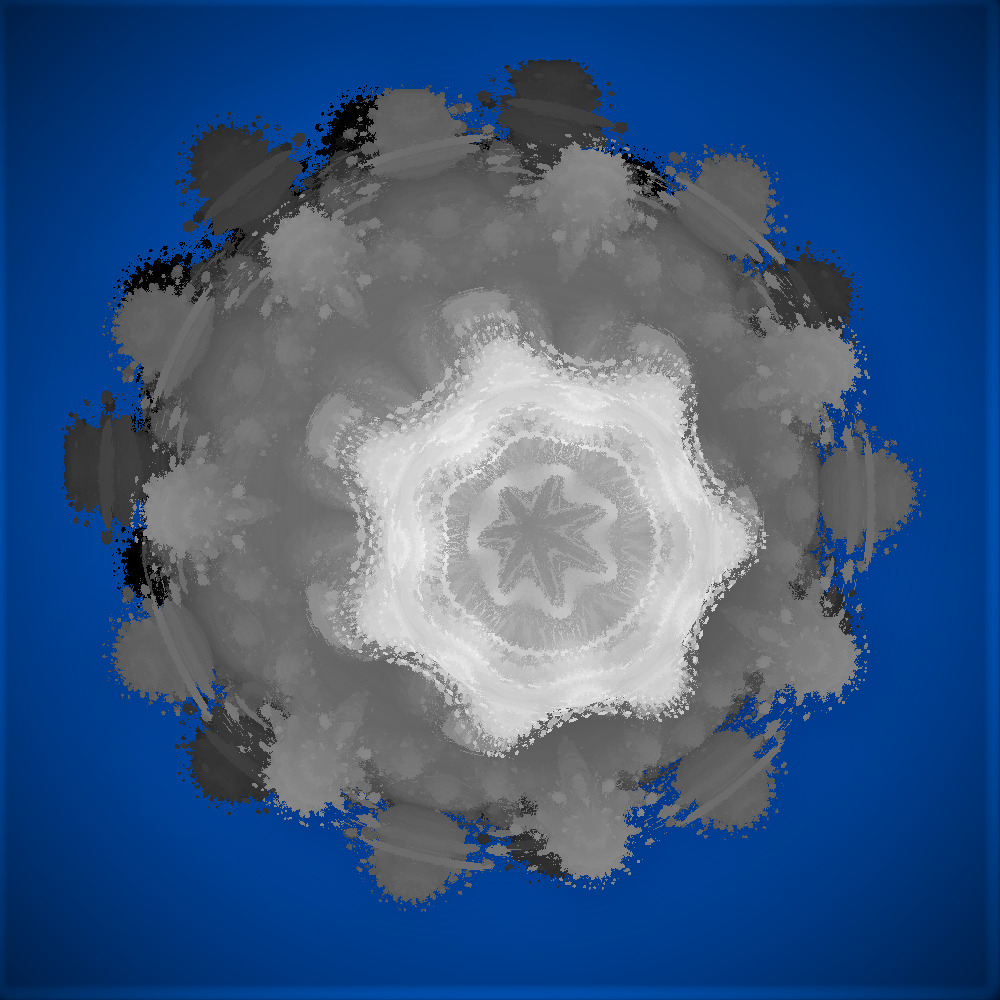}
	\caption{$\spher^{8}$ generated in Python.}
	\end{subfigure}
\hfill
	\begin{subfigure}{0.3\textwidth}
	\centering
	\includegraphics[width=3.5cm] 
	{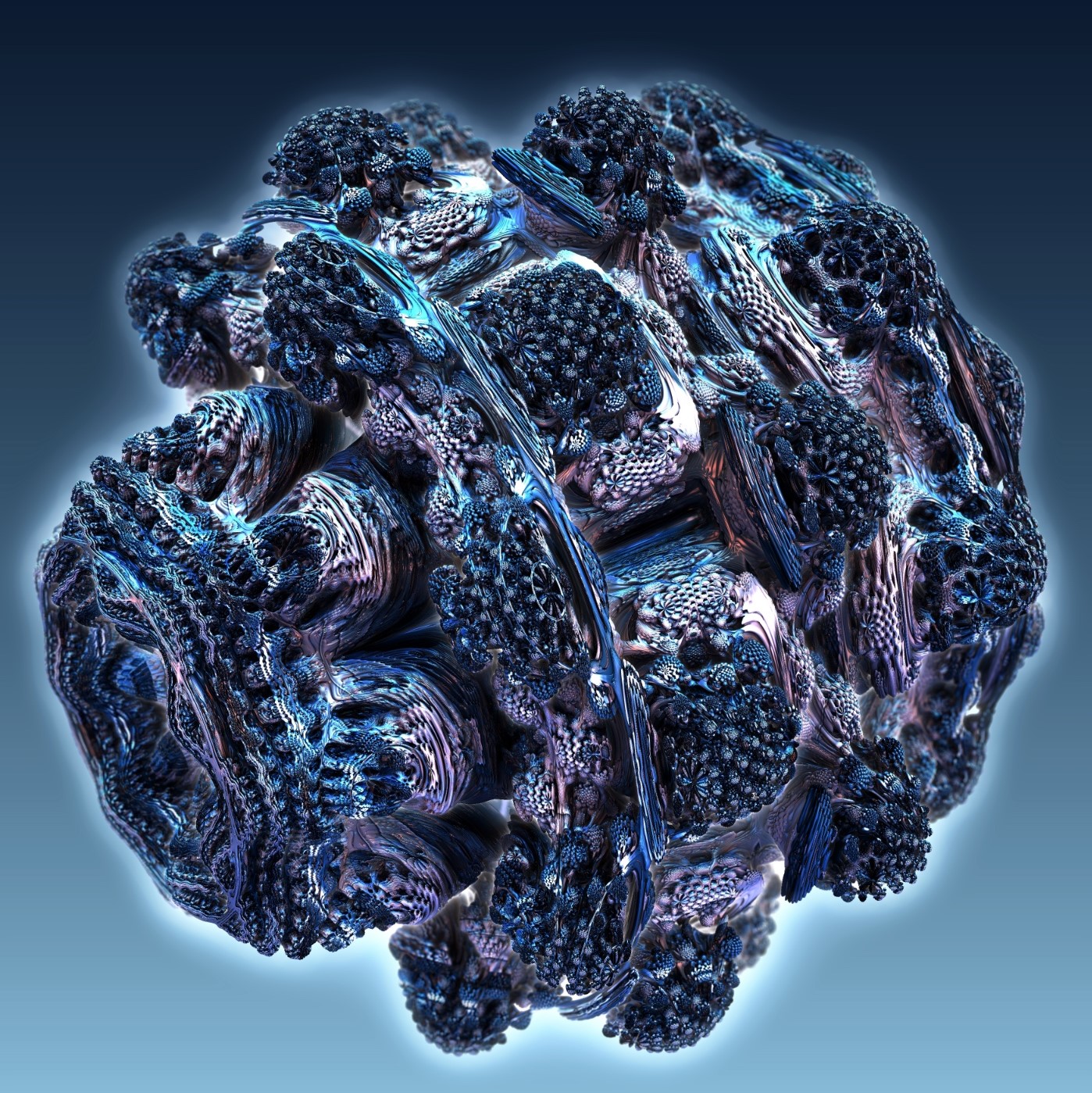}
	\caption{Mandelbulb generated with Mandelbulb 3D.}
	\end{subfigure}
\caption{\label{mandelbulb}Comparison between $\spher^{8}$ and the Mandelbulb.}
\end{figure}

\begin{figure}[h]
\centering
	\begin{subfigure}{0.3\textwidth}
	\centering
	\includegraphics[width=3.5cm]
	{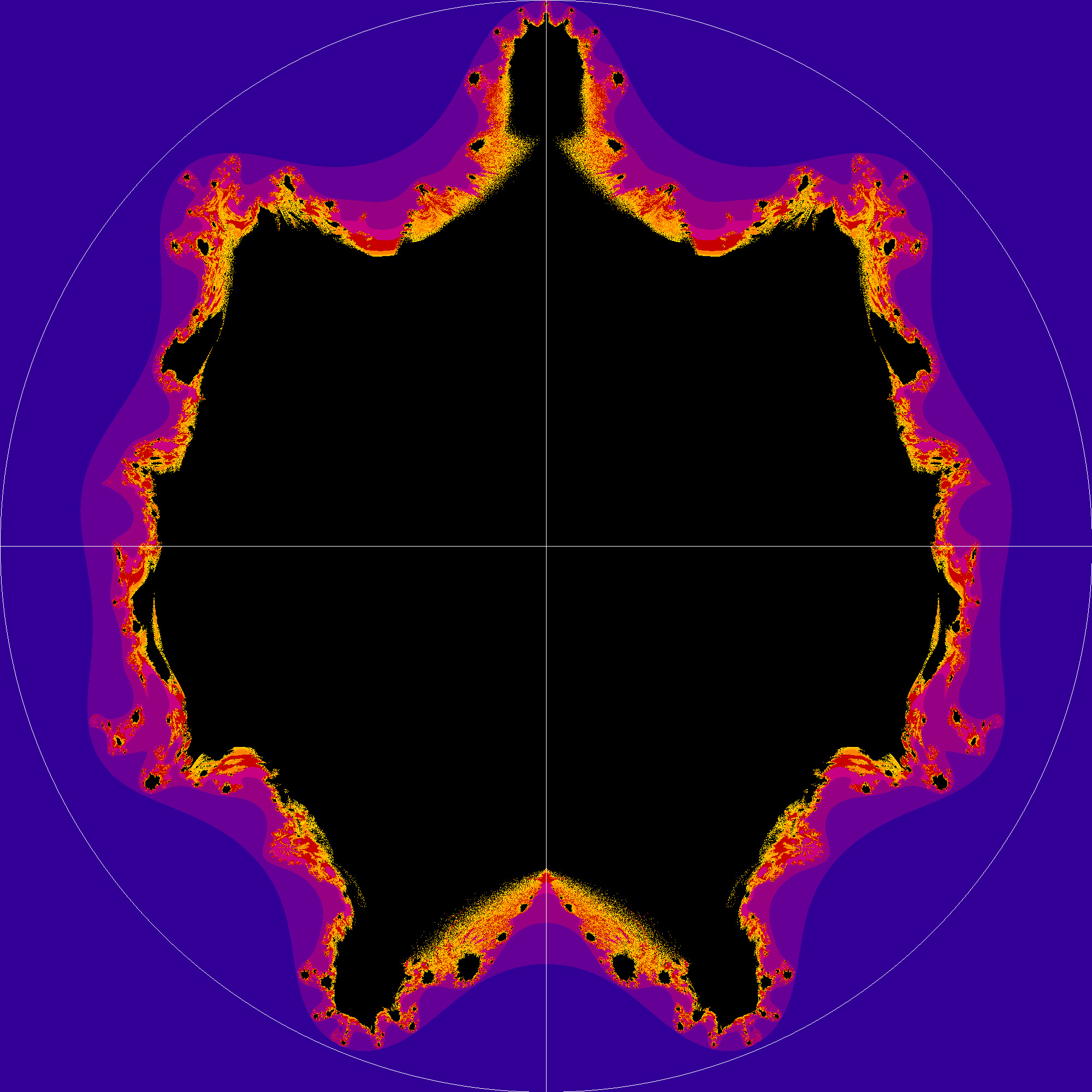}
	\caption{Slice when $x=0$.}
	\end{subfigure}
\hfill
\begin{subfigure}{0.3\textwidth}
	\centering
	\includegraphics[width=3.5cm]
	{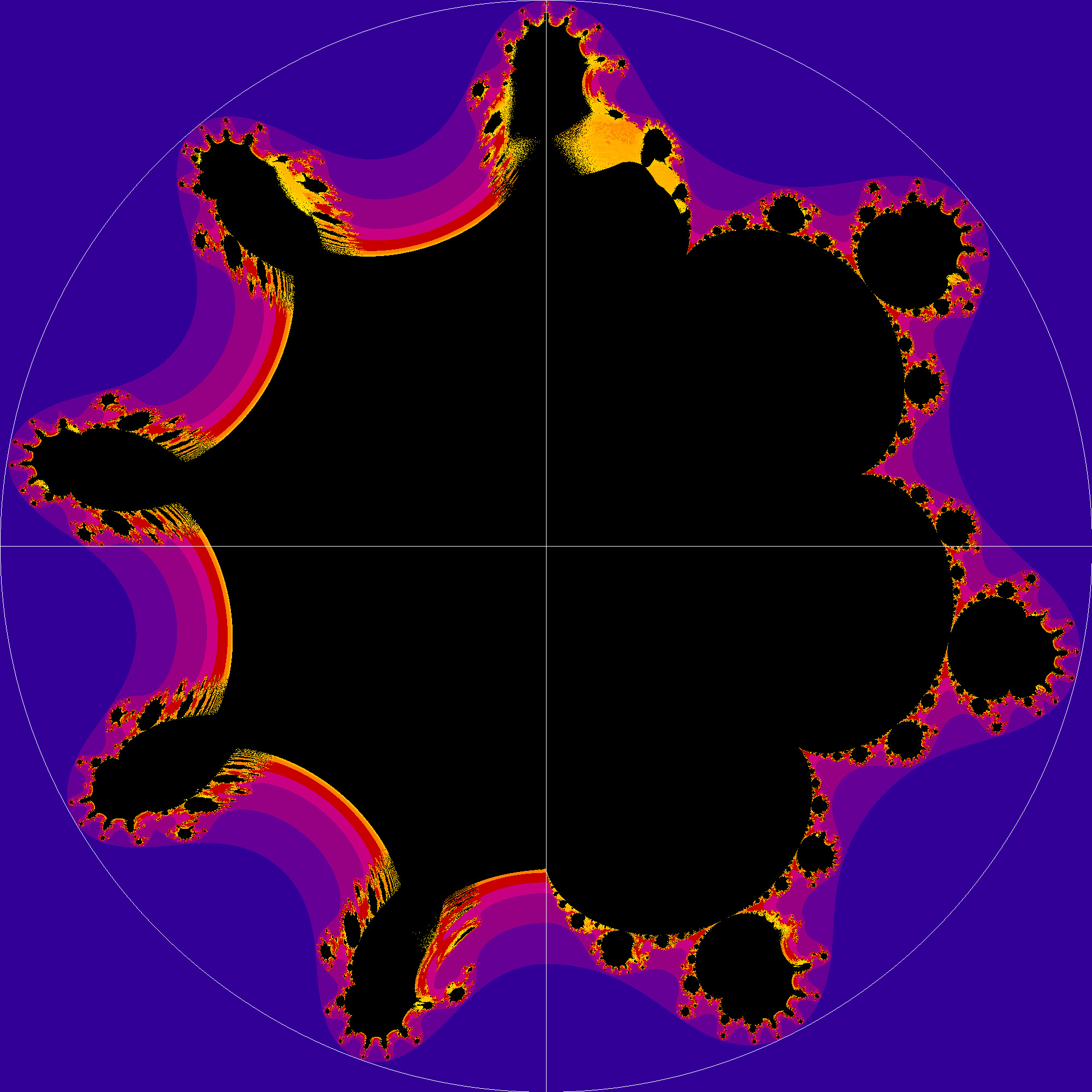}
	\caption{Slice when $y=0$.}
	\end{subfigure}
\hfill
	\begin{subfigure}{0.3\textwidth}
	\centering
	\includegraphics[width=3.5cm] 
	{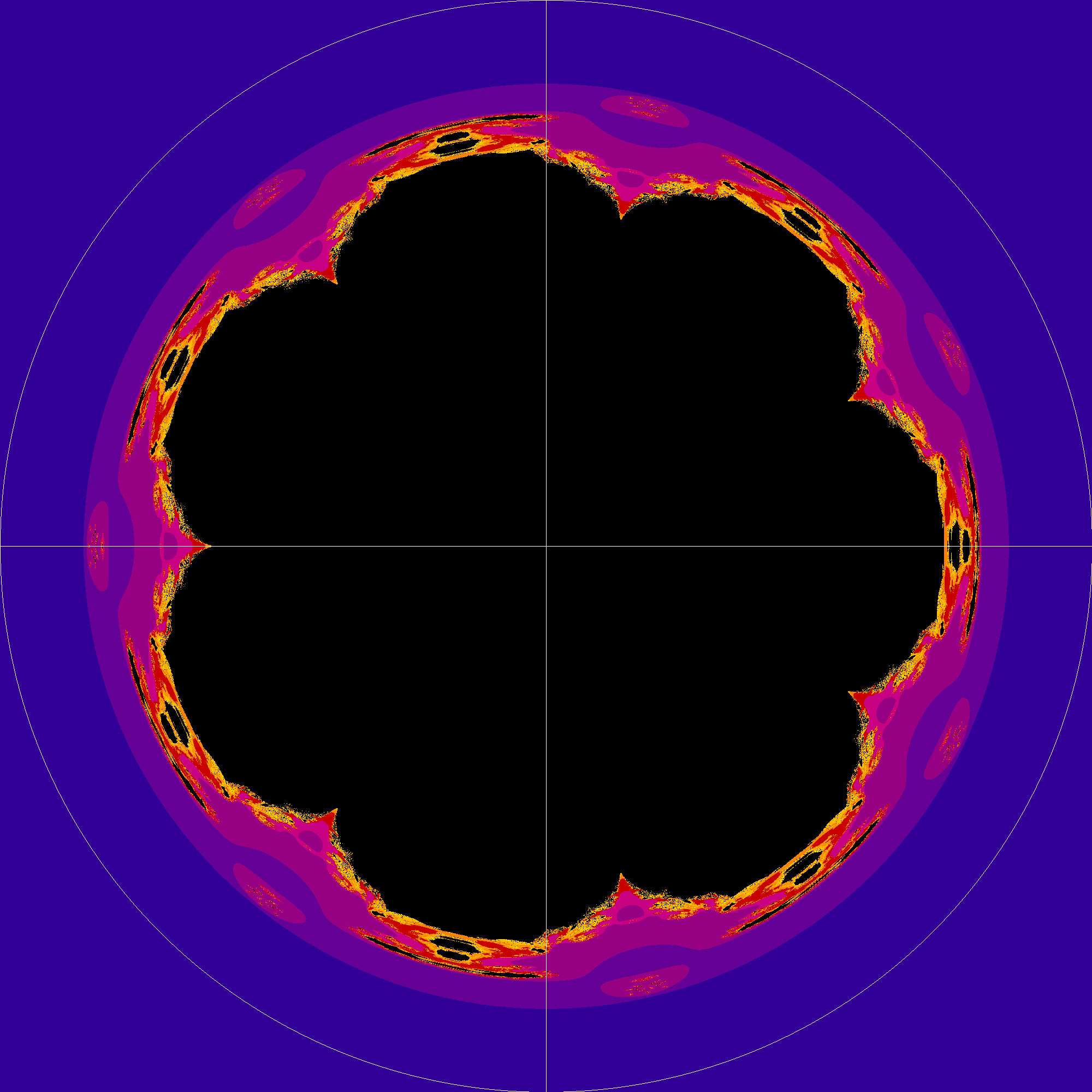}
	\caption{Slice when $z=0$.}
	\end{subfigure}
\caption{\label{mandelbulbcuts}Three slices of $\spher^{8}$ with divergence layers.}
\end{figure}

\begin{theorem}
A pure quaternion $c$ is in the Mandelbrot set $\spher^m$ if and only if $\|Q_{m,c}^{(n)}(0)\|\leq 2^{1/(m-1)}$ for all $n$.
\end{theorem}
\begin{proof}
The fact that if $\|Q_{m,c}^{(n)}(0)\|\leq 2^{1/(m-1)}$ for all $n$, then $c \in \spher^m$ is a direct consequence of the spherical Mandelbrot set $\spher^m$ definition. We now show that if $c \in \spher^m$, then $\|Q_{m,c}^{(n)}(0)\|\leq 2^{1/(m-1)}$ for all $n$. Let $c \in \spher^m$. Then, from Theorem \ref{mnorm}, we have $\|c\| \leq 2^{1/(m-1)}$. Suppose there exists an $n\in\mathbb{N}$ such that $\| Q^{(n)}_{m,c}(0)\| = 2^{1/(m-1)} + \delta$ for some $\delta >0$. Then, from Lemma \ref{mlemma}, $\| Q^{(n+k)}_{m,c}(0)\| \geq 2^{1/(m-1)} + (2m)^k\delta$ for all $k \geq 1$. Thus, $\| Q_{m,c}^{(n+k)}(0) \| \rightarrow \infty$ which means that $\| Q_{m,c}^{(n)}(0) \| \rightarrow \infty$ and that $ c \notin \spher^m$. We have a contradiction. Therefore, we have $\|Q_{m,c}^{(n)}(0)\| \leq 2^{1/(m-1)}$ for all $n$. 
\end{proof}

We found that the boundary is the same as the one for Multibrots. We use this boundary to generate the set using Python as showed in Figure \ref{mandelbulb}. Figure \ref{mandelbulbcuts} shows three slices of $\spher^8$ with divergence layers. It also shows the disc of radius $2^{1/7}$ centered at $0$ in white. We see that the set is contained in the disc as showed in Theorem \ref{mnorm}.

\subsection{Variations of the spherical Mandelbrot set}
The aim of this section is to introduce other generalizations of the Mandelbrot set that are variations of the spherical Mandelbrot set presented in the last section. Specifically, we use a variation of the spherical product when the angles $\phi$ and $\theta$ does not necessarily vary the same way. Let’s start with the following definitions where the minimal value for the multipliers of the angles are not required to be $2$ as for the previous results.
\begin{definition}
Let $a, b\in\mathbb{R}$. Consider a pure quaternion $q$ in its spherical representation. The product $\times_{s}^{(a,b)}$ is denoted: 
$$
q \times_{s}^{(a,b)} q := \rho^2 \big( i\sin(a\phi)\cos(b\theta)+j\sin(a\phi)\sin(b\theta)+k\cos(a\phi)\big).
$$
\end{definition}
We note that the variable $a$ is the multiplier of the angle $\phi$ when $b$ is the multiplier of the angle $\theta$.
\begin{figure}[htp]
\centering
\includegraphics[scale=0.05]{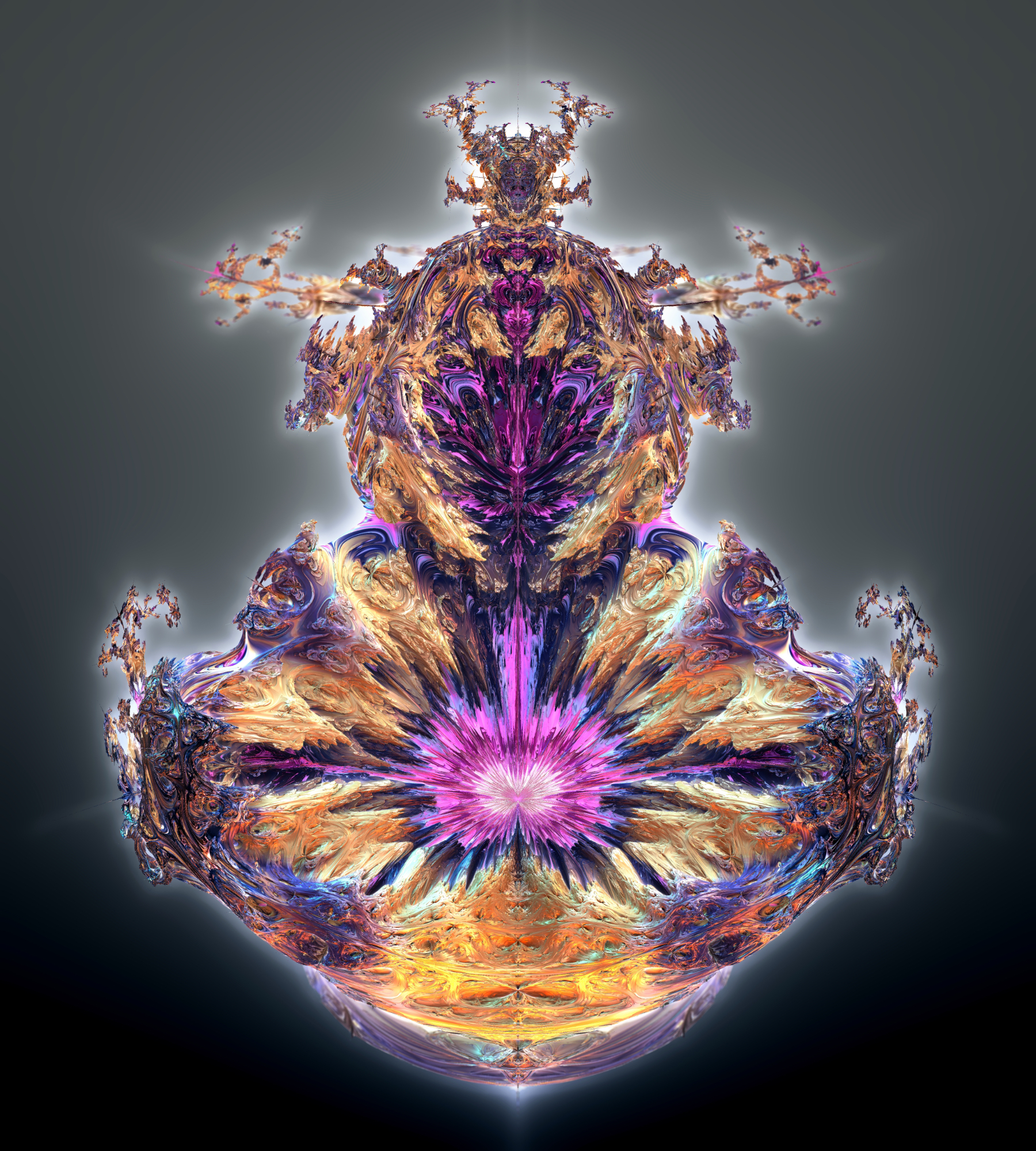}\hspace{10mm}
\includegraphics[scale=0.0665]{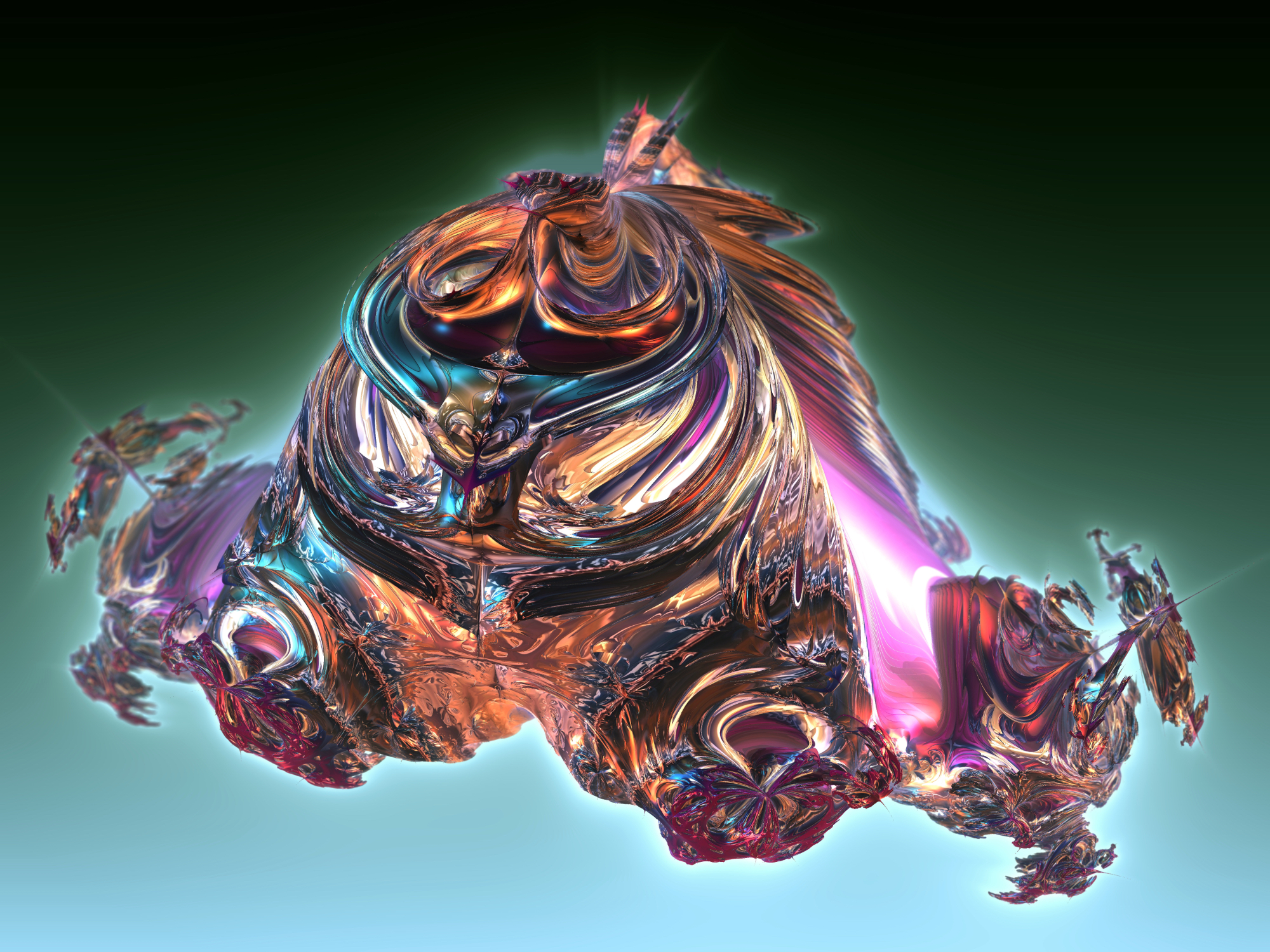}\hfill
\caption{\label{fig:Goldenbulb}The Goldenbulb for the power 2 and 3.}
\end{figure}

\begin{definition}
Let $a, b\in\mathbb{R}$. The Mandelbrot set $\spher^{(a,b)}$ is defined as
$$
\spher^{(a,b)}=\lbrace c\in\ \textnormal{Im}\mathbb{H}\ |\  \lbrace Q^{(n)}_c(0)\rbrace_{n\in \mathbb{N}}\ \text{is bounded}\rbrace,
$$
where $Q^{(n)}_c(q) = q \times_s^{(a,b)} q +c$.
\end{definition}
The last definition is very general. In this section, we will study specifically the case of $\spher^{(1,2)}$ and $\spher^{(2,1)}$. We note that the special case $\spher^{(\varphi,n)}$ where $\varphi$ is the Golden Ratio is called the \textit{Goldenbulb} for the power $n$ (see Figure \ref{fig:Goldenbulb} for $n=2,3$).

Now, we present the set $\spher^{(1,2)}$ that is showed in Figure \ref{fig:bulbic3d}. The following theorem shows that the slice when $z=0$ has the same dynamics as the Mandelbrot set. Figure \ref{coupesbulbiques} shows this slice generated with Python.
\begin{figure}[htp]
\centering
\includegraphics[scale=0.7]{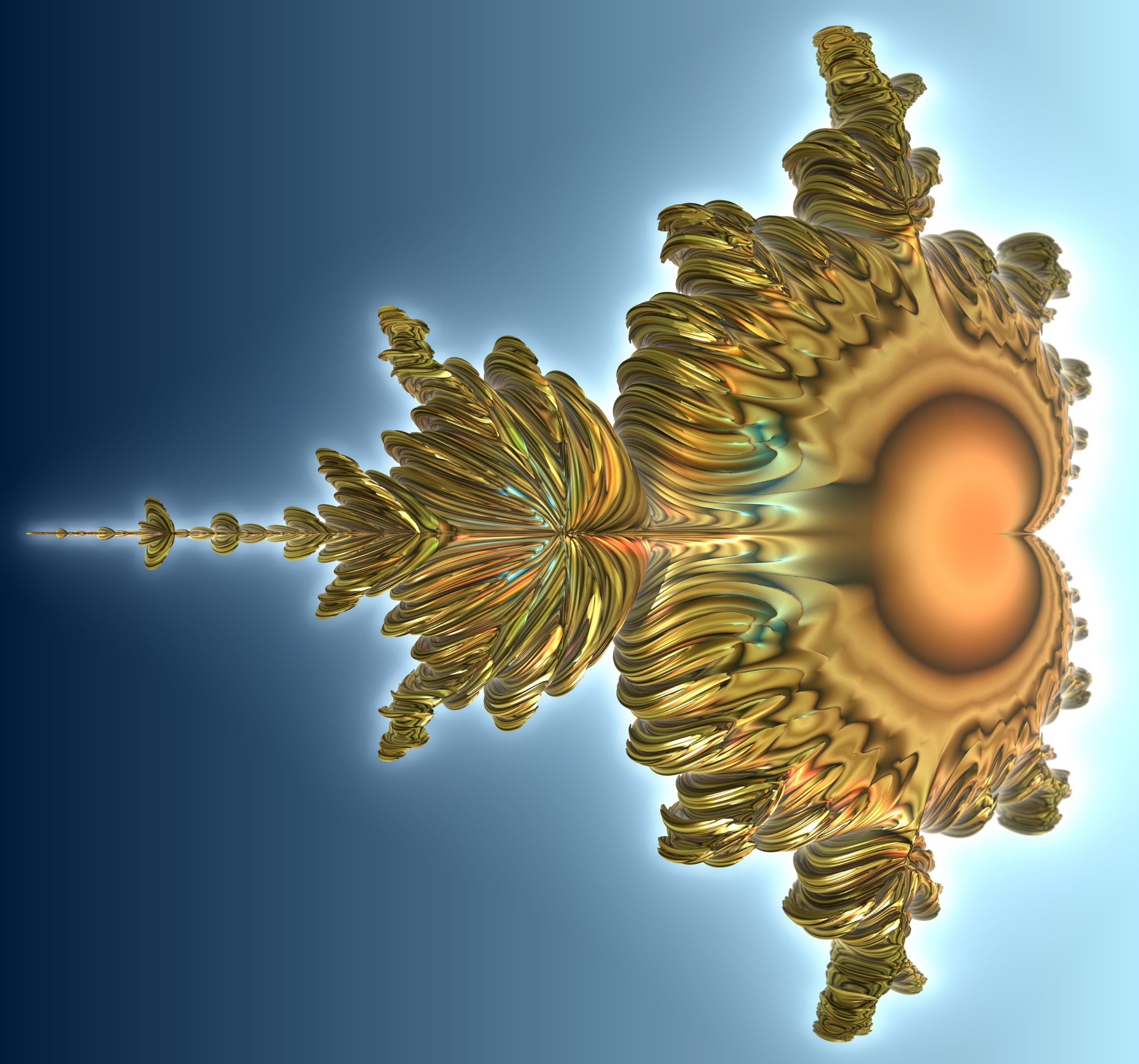}
\caption{\label{fig:bulbic3d}Representation of the Bulbic case, the set $\spher^{(1,2)}$.}
\end{figure}
\begin{theorem}
The slice when $z=0$ of the Mandelbrot set $\spher^{(1,2)}$ has the same dynamics as the Mandelbrot set $\mathcal{M}$ in the complex plane.
\end{theorem}
\begin{proof}
Consider $q$ a quaternion in the $xy$ plane and $c$ a complex number which have the same representation of an arbitrary point $(a,b)$ in their respective plane. Thus, we have $q=ai+bj$ and $c=a+bi$. We show that $Q_q^{(n)}(0)$ and $P_c^{(n)}(0)$ have also the same representation which means that
\begin{align*}
Q_q^{(n)}(0) = \rho_n(i\cos\theta_n+j\sin\theta_n) && \text{and} && P_c^{(n)}(0) = \rho_n(\cos\theta_n+i\sin\theta_n).
\end{align*}
We show it using induction on $n$. For $n=1$, we have, from the unique spherical representation of a pure quaternion,
\begin{align*}
Q^{(1)}_q(0) &= q \\
&= ai+bj\\
&= \rho_1(i\sin(\pi/2)\cos\theta_1+j\sin(\pi/2)\sin\theta_1 + k\cos(\pi/2))\\
&= \rho_1(i\cos\theta_1+j\sin\theta_1)
\end{align*}
where $\rho_1=\sqrt{a^2+b^2}$ and $\theta_1 = \arctantwo(b,a)$ with $0\leq\theta_1<2\pi$. Moreover, we have, using the unique polar representation of a complex number,
\begin{align*}
P^{(1)}_c(0) &= c \\
&= a+bi\\
&= \rho_1(\cos\theta_1+i\sin\theta_1)
\end{align*}
where $\rho_1=\sqrt{a^2+b^2}$ and $\theta_1 = \arctantwo(b,a)$ with $0\leq\theta_1<2\pi$. Thus, the statement is true for $n=1$, since we have
\begin{align*}
Q_q^{(1)}(0) = \rho_1(i\cos\theta_1+j\sin\theta_1) && \text{and} && P_c^{(1)}(0) = \rho_1(\cos\theta_1+i\sin\theta_1).
\end{align*}
Now, we suppose that the statement is true for $n=m$ and we show it is true for $n=m+1$. From the induction hypothesis, the unique spherical representation of $Q_q^{(m)}(0)$ is
$$
Q_q^{(m)}(0) = \rho_m(i\sin(\pi/2)\cos\theta_m+j\sin(\pi/2)\sin\theta_m + k\cos(\pi/2)).
$$
So, we have
\begin{align*}
Q^{(m+1)}_q(0) &= Q_q^{(m)}(0) \times_s^{(1,2)} Q_q^{(m)}(0) + q\\
&= \rho_m^2(i\sin(\pi/2)\cos 2\theta_m+j\sin(\pi/2)\sin 2\theta_m + k\cos(\pi/2))\\
& + \rho_1(i\cos\theta_1+j\sin\theta_1)\\
&= (\rho^2_m\cos 2\theta_m + \rho_1 \cos \theta_1)i + (\rho^2_m\sin 2\theta_m + \rho_1\sin\theta_1)j.
\end{align*}
Moreover, we have, using the De Moivre formula and the induction hypothesis,
\begin{align*}
P^{(m+1)}_c(0) &= [P^{(m)}_c(0)]^2 + c \\
&= [\rho_m(\cos\theta_m+i\sin\theta_m)]^2 + \rho_1(\cos\theta_1+i\sin\theta_1)\\
&= \rho_m^2(\cos 2\theta_m+i\sin 2\theta_m)+ \rho_1(\cos\theta_1+i\sin\theta_1)\\
&= (\rho^2_m\cos 2\theta_m + \rho_1 \cos \theta_1) + (\rho^2_m\sin 2\theta_m + \rho_1\sin\theta_1)i.
\end{align*}
We obtain that $Q^{(m+1)}_q(0)$ and $P_c^{(m+1)}(0)$ have the same components and therefore can be rewritten in unique representation in this way:
\begin{align*}
Q_q^{(m+1)}(0) &= \rho_{m+1}(i\cos\theta_{m+1}+j\sin\theta_{m+1})
\shortintertext{and}
P_c^{(m+1)}(0)& = \rho_{m+1}(\cos\theta_{m+1}+i\sin\theta_{m+1}).
\end{align*}
Thus, since the statement is true for $n=m+1$, it is also true for all $n \geq 1$. Therefore, since the iterates have the same representation in their respective plane, we conclude that the $z=0$ slice of $\spher^{(1,2)}$ is visually the same as the complex Mandelbrot set.
\end{proof}
\begin{figure}[ht]
\centering
	\begin{subfigure}{0.45\textwidth}
	\centering
	\includegraphics[width=5cm]
	{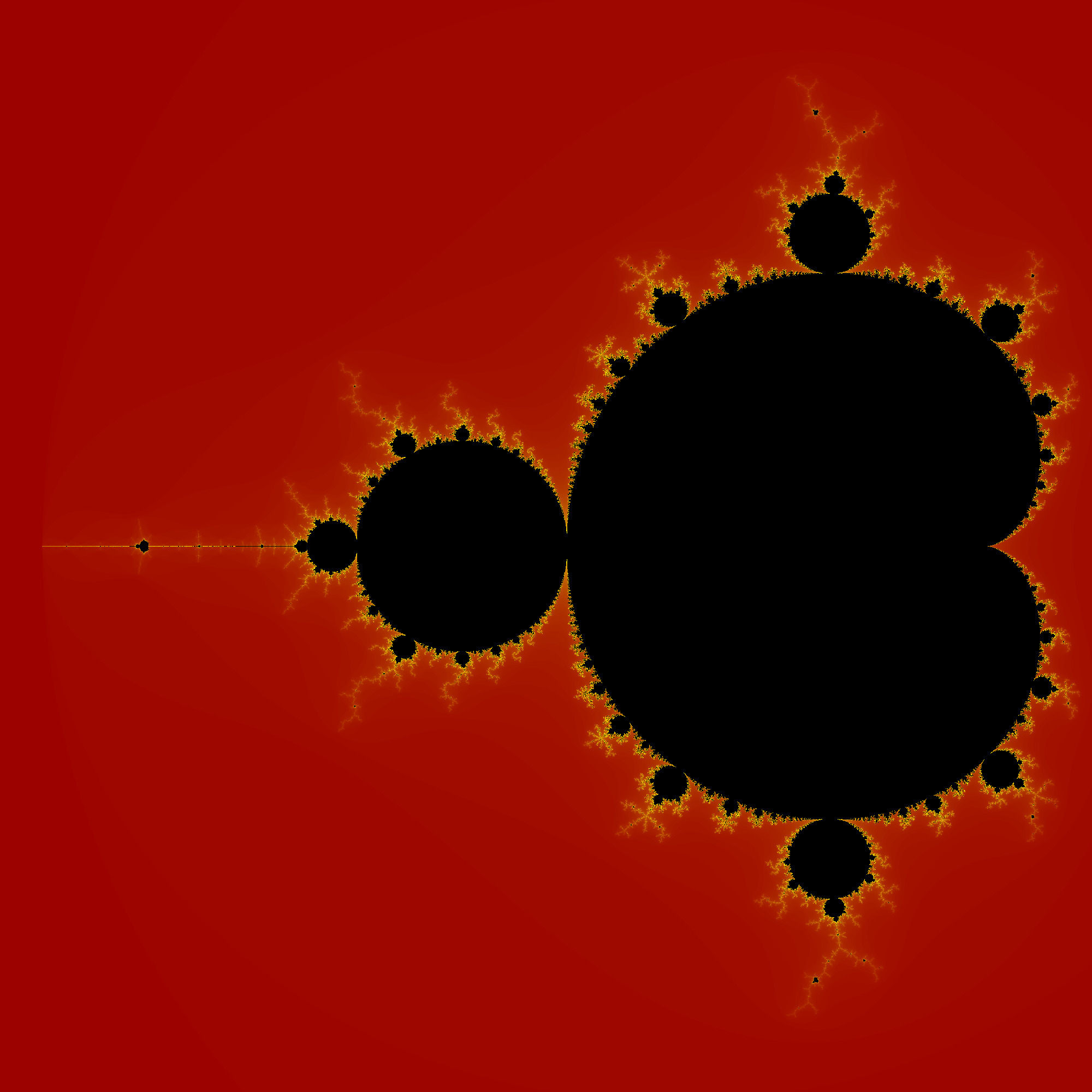}
	\caption{Generated in Python.}
	\end{subfigure}
\hfill
	\begin{subfigure}{0.45\textwidth}
	\centering
	\includegraphics[width=5cm] 
	{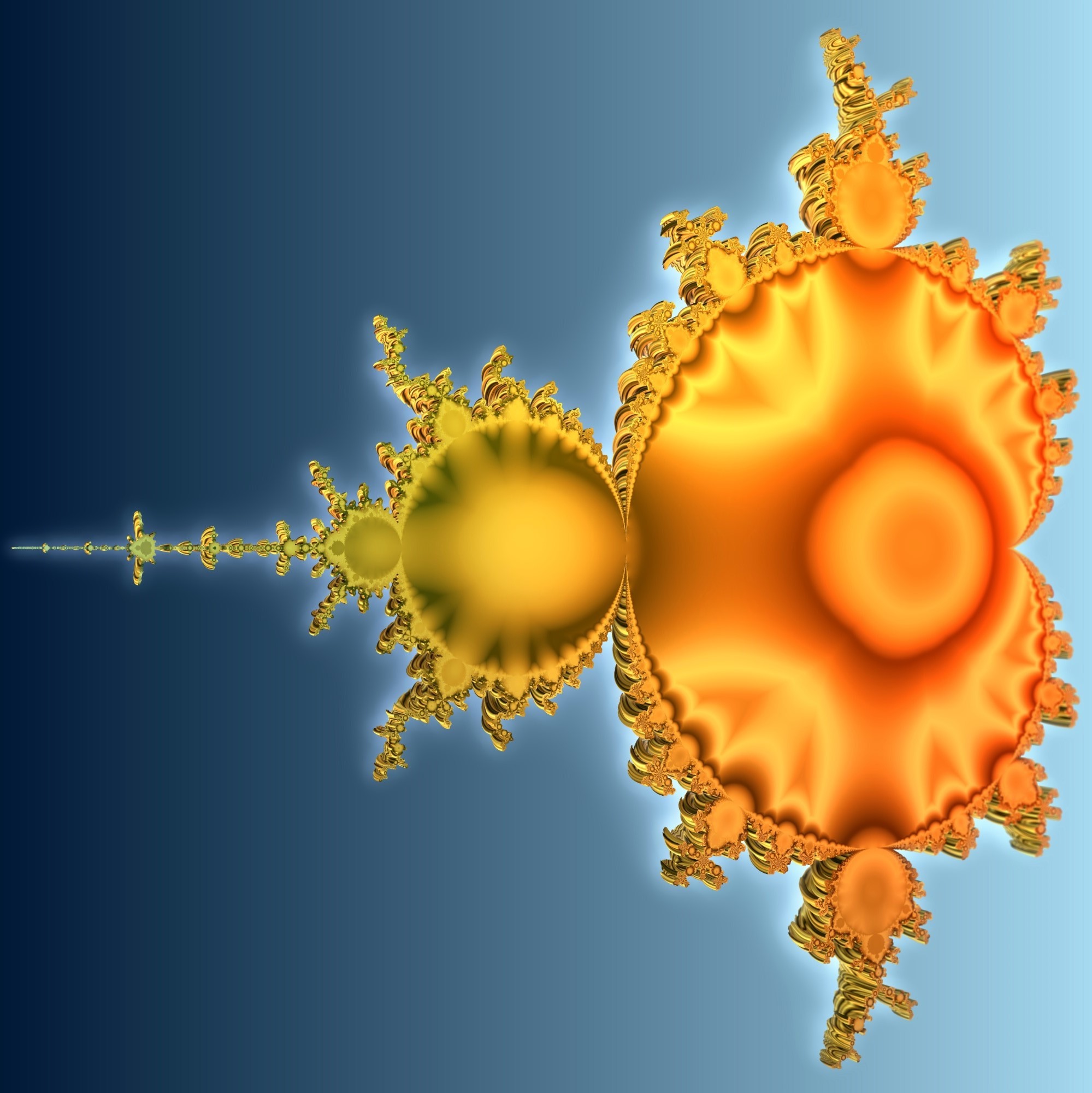}
	\caption{Generated with Mandelbulb 3D.}
	\end{subfigure}
\caption{\label{coupesbulbiques} Comparison of the $z=0$ slice of $\spher^{(1,2)}$ and the Bulbic Mandelbrot set.}
\end{figure}
The following theorem links the quaternionic rotation and the product $\times_s^{(1,2)}$. In fact, this product can be interpreted as a quaternionic product.
\begin{theorem}
\label{theo:link}
Let $q$ be a pure quaternion. Then,
$$
q \times_s^{(1,2)} q = \rho(\cos(\theta/2)+k\sin(\theta/2))q((\cos(\theta/2)-k\sin(\theta/2)),
$$
where $\rho$ and $\theta$ are respectively the radius and the angle of the spherical representation of $q$.
\end{theorem}
\begin{proof}
Consider $q$ a pure quaternion and its spherical representation
$$
q=\rho(i\sin\phi\cos\theta + j\sin\phi\sin\theta+k\cos\phi).
$$
So, replacing $q$ in the following equation
$$
\rho(\cos(\theta/2)+k\sin(\theta/2))q((\cos(\theta/2)-k\sin(\theta/2))
$$
we have
\noindent $$\rho^2[\cos(\theta/2)+k\sin(\theta/2)][i\sin\phi\cos\theta + 
j\sin\phi\sin\theta+k\cos\phi][(\cos(\theta/2)-k\sin(\theta/2)].$$

Now, by multiplying the terms and by using trigonometric identities we obtain 
$$\rho^2(i\sin\phi\cos(2\theta)+j\sin\phi\sin(2\theta) + k\cos\phi).$$ Thus, 
$$
q \times_s^{(1,2)} q = \rho(\cos(\theta/2)+k\sin(\theta/2))q(\cos(\theta/2)-k\sin(\theta/2)).
$$
\end{proof}
Theorem \ref{theo:rotation} and Theorem \ref{theo:link} allow us to conclude that the product $\times_s^{(1,2)}$ is equivalent to a rotation of $\theta$ around the vector $k$ when $q$ is unitary. 

The other variation case of the spherical Mandelbrot set that we present is the set $\spher^{(2,1)}$. The following theorem shows a link between this set and the quaternionic Mandelbrot set.
\begin{theorem}
\label{theo:spher21}
The Mandelbrot set $\spher^{(2,1)}$ has the same dynamics as the 3D slice $\mathcal{H}(1,i,j)$ of the quaternionic Mandelbrot.
\end{theorem}
\begin{proof}
Consider a pure quaternion $q=ak+bi+cj$ in the 3D space and a quaternion $p=a+bi+cj$. These two quaternions represent the same point $(a,b,c)$ in their respective space. We show that $Q_q^{(n)}(0)$ and $P_p^{(n)}(0)$ have also the same representation which means that
\begin{align*}
Q_q^{(n)}(0) & = \rho_n\left[k\cos\phi_n + \sin\phi_n\left( \frac{bi+cj}{\sqrt{b^2+c^2}} \right)\right]
\shortintertext{and}
 P_p^{(n)}(0) & =  \rho_n\left[\cos\phi_n + \sin\phi_n\left( \frac{bi+cj}{\sqrt{b^2+c^2}} \right)\right].
\end{align*}
We show it by mathematical induction on $n$. For $n=1$, we have, from the unique spherical representation of a pure quaternion,
\begin{align*}
Q^{(1)}_q(0) &= q \\
&= ak + bi+cj\\
&= \rho_1(i\sin\phi_1\cos\theta+j\sin\phi_1\sin\theta + k\cos\phi_1)
\end{align*}
where $\rho_1=\sqrt{a^2+b^2+c^2}$, $\phi_1=\arccos(a/\rho_1) $, $\displaystyle \cos\theta = \frac{b}{\sqrt{b^2+c^2}}$ and $\displaystyle \sin\theta = \frac{c}{\sqrt{b^2+c^2}}$. So, we have
$$
Q^{(1)}_q(0) = \rho_1\left[k\cos\phi_1 + \sin\phi_1\left( \frac{bi+cj}{\sqrt{b^2+c^2}} \right)\right].
$$
Moreover, we have, from the unique polar representation of a quaternion,
\begin{align*}
P^{(1)}_p(0) &= p \\
&= a+bi+cj\\
&= \rho_1\left[\cos\phi_1 + \sin\phi_1\left( \frac{bi+cj}{\sqrt{b^2+c^2}} \right)\right]
\end{align*}
where $\rho_1=\sqrt{a^2+b^2+c^2}$ and $\phi_1=\arccos(a/\rho_1) $. So, the statement is true for $n=1$ since we have
\begin{align*}
Q_q^{(1)}(0) & = \rho_1\left[k\cos\phi_1 + \sin\phi_1\left( \frac{bi+cj}{\sqrt{b^2+c^2}} \right)\right]
\shortintertext{and}
 P_p^{(1)}(0) & =  \rho_1\left[\cos\phi_1 + \sin\phi_1\left( \frac{bi+cj}{\sqrt{b^2+c^2}} \right)\right].
\end{align*}
Now, we suppose that the statement is true for $n=m$ and we show that it is true for $n=m+1$. From the induction hypothesis, we have
$$
Q_q^{(m)}(0)  = \rho_m\left[k\cos\phi_m + \sin\phi_m\left( \frac{bi+cj}{\sqrt{b^2+c^2}} \right)\right].
$$

By rewriting this quaternion in unique spherical representation, we obtain
$$
Q_q^{(m)}(0) = \rho_m\left( k\cos\phi_m + i\cos\theta\sin\phi_m + j\sin\theta\sin\phi_m \right)
$$
since $\displaystyle \cos\theta = \frac{b}{\sqrt{b^2+c^2}}$ and $\displaystyle \sin\theta = \frac{c}{\sqrt{b^2+c^2}}$.
Thus, we have
\begin{small}
\begin{align*}
Q^{(m+1)}_q(0) &= Q_q^{(m)}(0) \times_s^{(2,1)} Q_q^{(m)}(0) + q\\
& = \rho_m^2\left( k\cos(2\phi_m) + i\cos\theta\sin(2\phi_m) + j\sin\theta\sin(2\phi_m)\right)\\
& +   \rho_1 \left( k\cos\phi_1 + i\cos\theta\sin\phi_1 + j\sin\theta\sin\phi_1 \right)\\
& =  i \cos\theta (\rho_m^2\sin(2\phi_m)+\rho_1 \sin\phi_1) + j\sin\theta(\rho_m^2\sin(2\phi_m)+\rho_1 \sin\phi_1)\\
& + k (\rho_m^2 \cos(2\phi_m)+\rho_1\cos\phi_1)\\
& = (\rho^2_m\sin(2\phi_m)+\rho_1\sin\phi_1)(i\cos\theta+j\sin\theta) 
+ k (\rho^2_m\cos(2\phi_m)+\rho_1\cos\phi_1)\\
&= (\rho^2_m\sin(2\phi_m)+\rho_1\sin\phi_1)\left( \frac{bi+cj}{\sqrt{b^2+c^2}} \right)
+ k (\rho^2_m\cos(2\phi_m)+\rho_1\cos\phi_1).
\end{align*}
\end{small}%
Moreover, we have, using the De Moivre quaternionic formula (Theorem \ref{theo:demoivre}) and the induction hypothesis, 
\begin{small}
\begin{align*}
P^{(m+1)}_p(0) &= [P^{(m)}_p(0)]^2 + c \\
&= \rho_m^2\left[ \cos(2\phi_m) + \sin(2\phi_m)\left( \frac{bi+cj}{\sqrt{b^2+c^2}} \right) \right] \\ 
& + \rho_1\left[ \cos\phi_1 + \sin\phi_1\left( \frac{bi+cj}{\sqrt{b^2+c^2}} \right) \right]\\
& = (\rho^2_m\cos(2\phi_m)+\rho_1\cos\phi_1) + \left( \frac{bi+cj}{\sqrt{b^2+c^2}} \right)(\rho^2_m\sin(2\phi_m)+\rho_1\sin\phi_1).
\end{align*}
\end{small}%
We obtain that $Q^{(m+1)}_q(0)$ and $P_c^{(m+1)}(0)$ have the same components and therefore can be rewritten in unique representation in thy way:
\begin{align*}
Q_q^{(m+1)}(0) &= \rho_{m+1}\left[k\cos\phi_{m+1} + \sin\phi_{m+1}\left( \frac{bi+cj}{\sqrt{b^2+c^2}} \right)\right]
\shortintertext{and}
P_c^{(m+1)}(0) & = \rho_{m+1}\left[\cos\phi_{m+1} + \sin\phi_{m+1}\left( \frac{bi+cj}{\sqrt{b^2+c^2}} \right)\right].
\end{align*}
Thus, since the statement is true for $n=m+1$, it is also true for all $n \geq 1$. Therefore, since the iterates have the same representation in their respective space, we conclude that the slice $\mathcal{H}(1,i,j)$ and the set $\spher^{(2,1)}$ are visually the same.
\end{proof}
Since, we showed in Theorem \ref{theo:quatslice} that the slice $\mathcal{H}(1,i,j)$ visually represents a rotation of the Mandelbrot set around the real axis, we can conclude using Theorem \ref{theo:spher21} that the set $\spher^{(2,1)}$ also visually represents a rotation of the Mandelbrot set around an axis.
\section{Conclusion}
In this paper, we used spherical coordinates to generalize the Mandelbrot set in 3D. We presented several generalizations. One of them was visually the same as the so-called Mandelbulb. We bounded this set, so we could generate 2D cuts with divergence layers. We also established a link between a variation of the spherical Mandelbrot set and a 3D slice of the quaternionic Mandelbrot set.

In subsequent works, it could be interesting to generalize the Mandelbrot set in the same way we did in this paper, but using the geographic coordinates system (GCS). These coordinates are used like spherical coordinates, but the angles are restricted in this way: $-\pi/2 \leq \phi \leq \pi/2$ and $-\pi < \theta \leq \pi$. As we can see in \cite{Barrallo}, fractals generated with these coordinates are not the same as with the standard spherical coordinates even if the images for the power $8$ look identical. In fact, the morphological difference appears clearly for the power 2, where the GCS case seems to contain the Mandelbrot set. Another interesting avenue of research would be to generalize Julia sets with the spherical product.
\section{Funding} V. Boily would like to thank the ISM for the awards of graduate research grants.

\section{Aknowledgments} The authors are grateful to the Fractal Community, especially for all the programmers of the Mandelbulb 3D software.

\section{Conflicts of interest} The authors declare no conflict of interest.

\bibliographystyle{abbrv}
\bibliography{Article_Final_05}

\begin{thebibliography}{10}

\bibitem{Barrallo}
J.~Barrallo.
\newblock Expanding the {M}andelbrot {S}et into {H}igher {D}imensions.
\newblock Proceedings of Bridges 2010 : Mathematics, Music, Art, Architecture,
  Culture, pages 247--254. The Bridges Organisation: Kansas, MO, USA, 2010.

\bibitem{beardon}
A.~F. Beardon.
\newblock {\em Iteration of Rational Functions: Complex Analytic Dynamical
  Systems}, volume 132 of {\em Graduate Texts in Mathematics}.
\newblock Springer-Verlag New York, first edition, 1991.

\bibitem{bedding}
S.~Bedding and K.~Briggs.
\newblock Iteration of quaternion maps.
\newblock {\em International Journal of Bifurcation and Chaos}, 5(3):877--881,
  1995.

\bibitem{Boily}
V.~Boily.
\newblock Étude de l'ensemble de {M}andelbrot généralisé dans l'espace des
  quaternions.
\newblock Master's thesis, Universit{\'e} du Qu{\'e}bec {\`a}
  {Trois-Rivi{\`e}res}, Canada, 2022.

\bibitem{brouillettepariserochon}
G.~Brouillette, P.-O. Paris{\'e}, and D.~Rochon.
\newblock Tricomplex {Distance} {Estimation} for {Filled}-in {Julia} {S}ets and
  {Multibrot Sets}.
\newblock {\em International Journal of Bifurcation and Chaos}, 29(6), 2019.

\bibitem{BrouilletteRochon}
G.~Brouillette and D.~Rochon.
\newblock Characterization of the principal 3\uppercase{D} slices related to
  the multicomplex \uppercase{M}andelbrot set.
\newblock {\em Advances in applied Clifford algebras}, 29(39), 2019.

\bibitem{cheng}
J.~Cheng and J.-r. Tan.
\newblock Generalization of 3{D} {M}andelbrot and {J}ulia sets.
\newblock {\em Journal of Zhejiang University-SCIENCE A}, 8:134--141, 2007.

\bibitem{dang}
Y.~Dang, L.~H. Kauffman, and D.~J. Sandin.
\newblock {\em Hypercomplex Iterations: Distance Estimation and Higher
  Dimensional Fractals}.
\newblock World Scientific, 2002.

\bibitem{douady}
A.~Douady and J.~H. Hubbard.
\newblock It{\'e}ration des polyn{\^o}mes quadratiques complexes.
\newblock {\em C.R. Acad. Sci. Paris -- S{\'e}rie {I}, Math.}, 294:123--126,
  1982.

\bibitem{falconer}
K.~Falconer.
\newblock {\em Fractal geometry : {Mathematical} {Foundations} and
  {Applications}}.
\newblock Wiley, 2 edition, 2003.

\bibitem{GarantPelletier}
V.~Garant-Pelletier.
\newblock Ensembles de {Mandelbrot} et de {Julia} classiques,
  g{\'e}n{\'e}ralis{\'e}s aux espaces multicomplexes et th{\'e}or{\`e}me de
  {Fatou-Julia} g{\'e}n{\'e}ralis{\'e}.
\newblock Master's thesis, Universit{\'e} du Qu{\'e}bec {\`a}
  {Trois-Rivi{\`e}res}, Canada, 2011.

\bibitem{hanson}
A.~J. Hanson.
\newblock {\em Visualizing {Q}uaternions}.
\newblock The {M}organ {K}aufmann {S}eries in {I}nteractive 3{D} {T}echnology :
  {E}lsevier {I}nc., 2006.

\bibitem{kantor}
I.~Kantor and A.~Solodovnikov.
\newblock {\em Hypercomplex {Numbers}}.
\newblock Springer-Verlag, 1989.

\bibitem{katunin}
A.~Katunin.
\newblock {\em A concise {Introduction} to {Hypercomplex} {Fractals}}.
\newblock CRC Press, 2017.

\bibitem{koecher}
M.~Koecher and R.~Remmert.
\newblock {\em Graduate {Texts} in {Mathematics} : Numbers ({Chapter 7} and
  {Chapter 8})}.
\newblock Springer, 1991.

\bibitem{CNRS}
J.~Leys.
\newblock Mandelbulb.
\newblock \url{http://images.math.cnrs.fr/Mandelbulb.html?lang=fr}, Janvier
  2010.
\newblock Accessed 2021-09-07.

\bibitem{rochonmartineau}
{\'E}.~Martineau and D.~Rochon.
\newblock On a bicomplex distance estimation for the {Tetrabrot}.
\newblock {\em International Journal of Bifurcation and Chaos},
  15(9):3039--3050, 2005.

\bibitem{nylander}
P.~Nylander.
\newblock \url{http://www.bugman123.com/Hypercomplex/index.html}.
\newblock Accessed 2021-12-06.

\bibitem{parise}
P.-O. Paris{\'e}.
\newblock Les ensembles de {Mandelbrot} tricomplexes g{\'e}n{\'e}ralis{\'e}s
  aux polyn{\^o}mes $\zeta^p+c$.
\newblock Master's thesis, Universit{\'e} du Qu{\'e}bec {\`a}
  {Trois-Rivi{\`e}res}, Canada, 2017.

\bibitem{RochonParise}
P.-O. Paris{\'e} and D.~Rochon.
\newblock A study of dynamics of the tricomplex polynomial $\eta^p+ c$.
\newblock {\em Nonlinear Dynamics}, 82(1--2):157--171, Oct. 2015.

\bibitem{quilez}
I.~Quilez.
\newblock \url{https://www.iquilezles.org/}.
\newblock Accessed 2021-09-07.

\bibitem{calcul}
J.~Stewart.
\newblock {\em Calcul à plusieurs variables}.
\newblock Modulo, 2 edition, 2016.

\bibitem{VallieresRochon}
A.~Valli{\`e}res and D.~Rochon.
\newblock Relationship between the {M}andelbrot {A}lgorithm and the {P}latonic
  {S}olids.
\newblock {\em Mathematics}, 10(3)(482):1--17, 2022.

\bibitem{wang}
X.-Y. Wang and Y.-Y. Sun.
\newblock The general quaternionic {M-J} sets on the mapping $z\leftarrow
  z^{\alpha}+c\ (\alpha \in {N})$.
\newblock {\em Computers and Mathematics with Applications}, 53:1718--1732,
  2007.

\bibitem{white}
D.~White.
\newblock \url{https://www.skytopia.com/project/fractal/mandelbrot.html}.
\newblock Accessed 2021-12-06.

\end{thebibliography}

\end{document}